\documentclass[a4paper,10pt,reqno]{amsart}
\usepackage{amsfonts,amssymb,amscd,amsmath,latexsym,amsbsy,enumerate,a4wide, tikz,  ifthen, mathrsfs}

\usepackage{todonotes}
\usetikzlibrary{calc}
\usepackage{bm}

\usepackage[latin1]{inputenc}

\newtheorem{thm}{Theorem}[section]

\newtheorem{lem}[thm]{Lemma}

\newtheorem{prop}[thm]{Proposition}

\theoremstyle{definition}
\newtheorem{rmk}[thm]{Remark}

\numberwithin{equation}{section}

\def\de{\delta}

\def\Ga{\Gamma}

\def\R{\mathbb{R}}
\def\C{\mathbb{C}}
\def\N{\mathbb{N}}

\newcommand{\rFs}[5]{\,_{#1}F_{#2} \left( \genfrac{.}{.}{0pt}{}{#3}{#4};#5 \right)}

\newcommand{\SU}{\mathrm{SU}}
\newcommand{\U}{\mathrm{U}}

\newcommand{\D}{\displaystyle}

\title[]{Some bivariate stochastic models arising \\ from group representation theory}

\author{Manuel D. de la Iglesia}
\address{Manuel D. de la Iglesia\\
Instituto de Matem\'aticas, Universidad Nacional Aut\'onoma de M\'exico, Circuito Exterior, C.U., 04510, Ciudad de México, México.}
\email{mdi29@im.unam.mx}

\author{Pablo Rom\'an}
\address{Pablo Rom\'an, CIEM,
FaMAF, Universidad Nacional de C\'ordoba, Medina Allende s/n Ciudad
Universitaria, C\'ordoba, Argentina}
\email{roman@famaf.unc.edu.ar}
\date{\today}
\thanks{The work of the first author is partially supported by PAPIIT-DGAPA-UNAM grant IA100515 (México), UC MEXUS-CONACYT grant CN-16-84 and MTM2015-65888-C4-1-P (Ministerio de Economía y Competitividad, Spain), while the work of the second author is supported by the Radboud Excellence Fellowship, CONICET grant PIP 112-200801-01533, FONCyT grant PICT 2014-3452 and by SeCyT-UNC.
}

\date{\today}



\begin{document}

\maketitle

\begin{abstract}
The aim of this paper is to study some continuous-time bivariate Markov processes arising from group representation theory. The first component (level) can be either discrete (quasi-birth-and-death processes) or continuous (switching diffusion processes), while the second component (phase) will always be discrete and finite. The infinitesimal operators of these processes will be now matrix-valued (either a block tridiagonal matrix or a matrix-valued second-order differential operator). The matrix-valued spherical functions associated to the compact symmetric pair $(\SU(2)\times \SU(2), \mathrm{diag} \, \SU(2))$ will be eigenfunctions of these infinitesimal operators, so we can perform spectral analysis and study directly some probabilistic aspects of these processes. Among the models we study there will be rational extensions of the one-server queue and Wright-Fisher models involving only mutation effects.
\end{abstract}

\section{Introduction}

It is very well known that many important results of one-dimensional stochastic processes can be obtained by using spectral methods. In particular, for Markov processes, many probabilistic aspects can be analyzed in terms of the (orthogonal) eigenfunctions and eigenvalues of the infinitesimal operator associated with the Markov process. In a series of papers in 1950-1960, S. Karlin and J. McGregor studied random walks and birth-and-death processes by using orthogonal polynomials (see \cite{KMc2}--\cite{KMc6}). Since the one-step transition probability matrix of the random walk or the infinitesimal operator of the birth-and-death process are tridiagonal matrices, it is possible to apply the spectral theorem to find the corresponding Borel measure associated with the process. With this measure it is easier to study the transition probabilities, the invariant measure or the behavior of the states of the process. Many other authors like M. Ismail, G. Valent, H. Dette, D. P. Maki or E. van Doorn, to mention a few, have studied this connection and other probabilistic aspects (see e.g. \cite{Dette, ILMV, Mak,vD2,vD3}). As for diffusion processes, it is also possible to use spectral methods, but now applied to second-order differential operators. Many authors like H. McKean, J. F. Barrett, D. G. Lampard, E. Wong or more recently D. Bakry, O. Mazet and B. Griffiths have studied this connection (see e.g. \cite{BW, BM, BL, Gri, IMcK, KT2, McK, WT}). Prominent examples are the Orstein-Uhlenbeck process, population growth models or Wright-Fisher models. For a brief account of the subject and other relations between stochastic processes and orthogonal polynomials, see \cite{S}.

A natural extension in this direction are bivariate Markov processes with discrete and finite second component. Now the state space is two-dimensional of the form $\mathcal{S}\times\{1,2,\ldots,N\}$, where $\mathcal{S}\subseteq\mathbb{R}$ is either a discrete set or a continuous interval, and $N$ is a positive integer. The first component is usually called the \emph{level}, while the second one is called the \emph{phase}. If $\mathcal{S}$ is discrete these processes are typically called \emph{quasi-birth-and-death processes} (see \cite{LaR, Neu}), while if $\mathcal{S}$ is a continuous real interval, they are called \emph{switching diffusion processes} (see \cite{MY, YZ}). They key point to study spectral methods of these processes will be the theory of matrix-valued orthogonal polynomials. In the last few years many progresses have been made in this direction. For discrete-time quasi-birth-and-death processes the extension of the Karlin-McGregor formula was given independently in \cite{DRSZ, G2}, while for continuous-time in \cite{DR}. For switching diffusion processes see \cite{dI2}.

A natural source of examples comes from group representation theory. There is a close relationship between special functions and harmonic analysis on groups that has been worked out for various classes of groups. E. Cartan and H. Weyl linked the classical theory of spherical harmonics with that of group representations showing that spherical harmonics arise naturally from the study of functions on the $n$-dimensional sphere $S^n = \mathrm{SO}(n + 1)/\mathrm{SO}(n)$. More generally, it is well known that the zonal spherical functions associated to real compact symmetric spaces can be realized as Jacobi polynomials. The link between zonal spherical functions and orthogonal polynomials has a matrix-valued analogue that was first investigated in \cite{GPT1} for the compact symmetric pair $(G,K)=(\SU(3),\U(2))$. The matrix-valued spherical functions are related to an auxiliary function which is an eigenfunction of a matrix-valued differential operator related to the Casimir operator of the group $G$ and that is given explicitly. A probabilistic interpretation for this case is given in \cite{GdI2} and is extended in \cite{GPT3}. An alternative approach to relate matrix-valued spherical functions and matrix-valued orthogonal polynomials is given in \cite{KvPR, KvPR2, HvP, vPR}, where more general families of symmetric pairs $(G,K)$ are treated. In this construction, one obtains a family of matrix-valued functions $\Psi_n$, together  with a matrix-valued differential operator $\Omega$, for which the functions $\Psi_n$ are eigenfunctions.
The first of these functions, $\Psi_0$, turns out to be invertible, and the sequence $P_n=\Psi_n\Psi_0^{-1}$ is a sequence of matrix-valued orthogonal polynomials with respect to an appropriate weight function which are eigenfunctions of a matrix-valued hypergeometric operator as in \cite{TiraPNAS}. 

The bispectral property of these examples will give us naturally a block tridiagonal Jacobi matrix (or a three-term recurrence relation) and a matrix-valued second-order differential operator, along with their eigenfunctions and eigenvalues. After appropriate conjugations it will be possible to transform these operators into infinitesimal operators of bivariate Markov processes. From the block tridiagonal Jacobi matrix we will get the infinitesimal operator of a continuous-time level-dependent quasi-birth-and-death process, while from the matrix-valued second-order differential operator we will get a switching diffusion process. The structure of the group will divide both processes into two independent processes, which will be studied in detail. For simplicity, we will focus on the lowest dimensional cases.

The structure of the paper goes as follows. In Section \ref{SEC2} we will give a brief account of matrix-valued spherical functions, focusing on the example for the pair $(G,K)=(\SU(2)\times \SU(2), \mathrm{diag} \, \SU(2))$ studied in \cite{KvPR,KvPR2} and the one-parameter extension given in \cite{KdlRR, vPR}. The second-order differential operator, three-term recurrence relation, weight matrix, norms and other structural formulas to transform the operators into operators with stochastic interpretation will be given. The reader interested exclusively in the stochastic models could skip this section and go directly to Sections \ref{SEC3} and \ref{SEC4}. In Section \ref{SEC3} we will study in detail the $3\times3$ case and we will use the spectral analysis to study several probabilistic aspects. From the block tridiagonal Jacobi matrix we will get two birth-and-death models. The first one is a regular birth-and-death process, while the second one is a continuous-time quasi-birth-and-death process with two phases (there are very few examples in the literature in this direction). Both can be viewed as rational extensions of the one-server queue with one free parameter. From the second-order differential operator we will get two diffusion models. The first one is a regular diffusion process with killing, while the second one is a switching diffusion process with two phases. Both can be viewed as extensions of the Wright-Fisher model involving only mutation effects. Finally, in Section \ref{SEC4} we will give some remarks about the $5\times5$ case, especially for the second-order differential operator. In this case we will get two models, a switching diffusion process with three phases, and a switching diffusion process with two phases \emph{with killing}. The spectral analysis of this last process appears to be new.

\section{Spherical functions and differential operators}\label{SEC2}

In this section $E_{ij}$ will denote the matrix with 1 at the entry $(i,j)$ and 0 elsewhere ($i,j\geq0$). Additionally we will use the following $N\times N$ diagonal matrices
\begin{equation}\label{JJ}
J=\sum_{i=0}^{N-1}(N-1-i)E_{ii},\quad \breve J=(N-1)I-J=\sum_{i=0}^{N-1}iE_{ii},
\end{equation}
and the nilpotent matrix of order $N$
\begin{equation}\label{AA}
A=\sum_{i=0}^{N-2}E_{i,i+1}.
\end{equation}
For any matrix $M\in\mathbb{C}^{N\times N}$, $M^*$ will denote the conjugate transpose of $M$. Also $I_N$ will denote, as usual, the identity matrix of dimension $N\times N$.

\subsection{Matrix-valued spherical functions}
Here we discuss the family of matrix-valued spherical functions given in \cite{KvPR,KvPR2} for the pair $(G,K)=(\SU(2)\times \SU(2), \mathrm{diag} \, \SU(2))$ and the one-parameter extension \cite{KdlRR, vPR}. For each $\ell\in \mathbb{N}$, if we let $N=2\ell+1$, it was shown in \cite{KvPR,KvPR2} that
there exists a family of $\mathbb{C}^{N\times N}$-valued functions $\{\Psi_n: n\in \mathbb{N}_0\}$, defined on the interval $[0,1]$. The family is constructed
by means of the spherical functions associated to $(G,K)$. All the properties of the spherical functions, like e.g. orthogonality relations, being eigenfunctions of differential operators, can be translated into properties of the functions $\Psi_n$. This family has a one parameter extension  $\{\Psi_n^{(\nu)}\}_{n\geq0}$ given in \cite{KdlRR,vPR}. The functions $\Psi_n^{(\nu)}$ satisfy the matrix-valued differential equation
\begin{equation}
\label{eqOmg}
\Omega^{(\nu)}  \Psi_n^{(\nu)}(y)= y(1-y)\frac{d^2 \Psi_n^{(\nu)}(y)}{dy^2} + a^{(\nu)}(y) \frac{d \Psi_n^{(\nu)}(y)}{dy} + F^{(\nu)}(y)\Psi_n^{(\nu)}(y) = \Psi_n^{(\nu)}(y)\,\Lambda_n^{(\nu)},
\end{equation}
where $a^{(\nu)}(y)=1/2+\nu-y(2\nu+1)$ and 
\begin{align}
\label{eq:Fnu_SU2}
F^{(\nu)}(y)&=\ell(\ell+2)-(\nu-1)(2\ell+\nu+1)-\frac{1}{2y(1-y)}\left[\ell(\nu-1)(1-2y)^2+\ell+J\breve{J}\right] \\
\nonumber&\qquad+J\breve{J}+\frac{1-2y}{4y(1-y)} \left(\breve{J}A^*(J+\nu-1)+JA(\breve{J}+\nu-1)\right),
\end{align}
where $J, \breve{J}$ and $A$ are given by \eqref{JJ} and \eqref{AA}.
%
%
%

\subsection{Matrix-valued orthogonal polynomials}
Matrix-valued spherical functions are closely related to matrix-valued orthogonal polynomials. In fact we have
$$
\Psi^{(\nu)}_n(y)=\left[\Psi_0^*(y)P^{(\nu)}_n(y)\right]^*,
$$
where $\Psi_0(y)$ is independent of $\nu$ and $(P_n^{(\nu)})_n$ is a family of monic matrix-valued orthogonal polynomials satisfying
\begin{equation}\label{defmonic}
\begin{split}
&\qquad \qquad \int_{0}^1  P_n^{(\nu)}(y) \, W^{(\nu)}(y)\,\bigl(P_m^{(\nu)}(y)\bigl)^\ast \, dy \, = \, \de_{nm} \|P_n^{(\nu)}\|^2_{W^{(\nu)}},
\end{split}
\end{equation}
where $\|P_n^{(\nu)}\|^2_{W^{(\nu)}}$ is the matrix-valued norm of the monic polynomial $P_n^{(\nu)}$ and it is given by the diagonal matrix with entries
\begin{gather}\label{NormsG}
\bigl( \|P_n^{(\nu)}\|^2_{W^{(\nu)}}\bigr)_{k,k} = \frac{\sqrt{\pi}}{2\cdot4^n}\,  \frac{\Ga(\nu+1/2)}{\Ga(\nu+1)}
\frac{\nu(2\ell+\nu+n)}{\nu+n} \frac{k!\, (2\ell-k)!\, (n+\nu+1)_{2\ell}}{(2\ell)!\, (n+\nu+1)_k (n+\nu+1)_{2\ell-k}} \\
\nonumber\qquad\qquad\qquad\times \frac{n!\, (\ell+1/2+\nu)_n (2\ell+\nu)_n(\ell+\nu)_n}{(2\ell+\nu+1)_n(\nu+k)_n(2\ell+2\nu+n)_n(2\ell+\nu-k)_n},
\end{gather}
and the weight matrix is given by
\begin{align}
\label{WWnu}W^{(\nu)}(y)&=\frac{4^{\nu-\ell}(\nu+\ell)_{\ell+1}}{2(\nu+1/2)_\ell}\, [y(1-y)]^{\nu-1/2} \, \left(\Psi_0(y)\right)^\ast T^{(\nu)} \Psi_0(y),\\
\nonumber T^{(\nu)}_{ij}&=\delta_{ij} \, \binom{2\ell}{i} \frac{(\nu)_i}{(\nu+2\ell-i)_i}.
\end{align}
Observe that the diagonal entries of $T^{(\nu)}_{ij}$ correspond up to a constant to the nodes of the beta-binomial distribution ($\alpha=\beta=\nu$). Note also that the $\nu$-dependence on the weight matrix is only located in the scalar weight  $\left[y(1-y)\right]^{\nu-1/2}$ and the constant diagonal matrix $T^{(\nu)}$. The function $\Psi_0(y)$ is the building block of the orthogonality measure and has been calculated explicitly in \cite{KvPR}. A nice compact formula for $\Psi_0(y)$ is given in \cite{vPR}. Let $K$ be the constant matrix with entries
$$K_{i,j}=K_j(i)=K_j(i,1/2,2\ell),$$
where $K_n(x,p,N)$ are the Krawtchouk polynomials, see e.g. \cite[\S 1.10]{KoekS}. Then we have
\begin{equation}
\label{def:Phi_0}
	\Psi_0(y)=K M \Upsilon(y)K^*,
\end{equation}
where $\Upsilon, M$ are the diagonal matrices
$$\Upsilon(y)_{jj} = (-1)^{\frac{3j}{2}} y^{\frac{j}{2}}(1-y)^{\frac{2\ell-j}{2}} ,\qquad M_{jj}= \binom{2\ell}{j}.$$	

Since the spherical functions $\Psi_n^{(\nu)}$ are eigenfunctions of $\Omega^{(\nu)}$, the matrix-valued
orthogonal polynomials $(P_n^{(\nu)})^*$ are eigenfunctions of the differential operator $ \Psi_0^{-1}\Omega^{(\nu)} \Psi_0$ which is explicitly given by
\begin{equation}\label{diffop}
D^{(\nu)}=y(1-y)\partial_y^2+(C+\nu-y(2\ell+2\nu+1))\partial_y+V+(\nu-1)(2\ell+\nu+1),\quad\partial_y=\frac{d}{dy},
\end{equation}
where
$$C=\frac{2\ell+1}{2}-\frac{1}{2}(A^\ast J+A\breve{J}),\qquad V=J\breve{J},
$$
and $J, \breve{J}$ and $A$ are given by \eqref{JJ} and \eqref{AA}.
Moreover, the operator $D^{(\nu)}$ is symmetric with respect to $W^{(\nu)}$. The eigenvalue for $D^{(\nu)}$ (and $\Omega^{(\nu)}$ in \eqref{eqOmg}) is
\begin{equation}\label{eigvalue}
\Lambda^{(\nu)}_n=-n(n-1)-n(2\ell+2\nu+1)+V+(\nu-1)(2\ell+\nu+1).
\end{equation}

Additionally the monic matrix-valued orthogonal polynomials $P_n^{(\nu)}$ satisfy a three-term recurrence relation of the form
\begin{equation}\label{TTRRPn}
yP_n^{(\nu)}(y)=P_{n+1}^{(\nu)}(y)+B_n^{(\nu)}P_n^{(\nu)}(y)+C_n^{(\nu)}P_{n-1}^{(\nu)}(y),\quad n\geq1,
\end{equation}
where the coefficients $B_n^{(\nu)}$ and $C_n^{(\nu)}$ are given by
\begin{align*}
B_n^{(\nu)}&=\frac{1}{2}-\frac{1}{4}J(J+\nu-1)\left[(J+n+\nu-1)(J+n+\nu)\right]^{-1}A\\
&\qquad-\frac{1}{4}\breve{J}(\breve{J}+\nu-1)\left[(\breve{J}+n+\nu-1)(\breve{J}+n+\nu)\right]^{-1}A^*,\quad n\geq0,
\end{align*}
and
\begin{align*}
C_n^{(\nu)}&=\frac{n(n+\nu-1)(2\ell+n+\nu)(2\ell+n+2\nu-1)}{16}\times\\
&\qquad\times\left[(J+n+\nu-1)(J+n+\nu)(\breve{J}+n+\nu-1)(\breve{J}+n+\nu)\right]^{-1},\quad n\geq1,
\end{align*}
where $J, \breve{J}$ and $A$ are given by \eqref{JJ} and \eqref{AA}.

\subsection{The function $S$}
In this subsection we turn the differential operator $\Omega^{(\nu)}$ into a differential operator which has a form that allows for a probabilistic interpretation by conjugating with a matrix-valued function. The appropriate function is given by  a diagonal matrix whose diagonal entries are those  of the $\ell$-th column of $\Psi_0(y)$. We assume $\ell\in \mathbb{N}$, so that $2\ell+1$ is odd. The $\ell$-th column (and the $\ell$-th row) of the matrix $\Psi_0(y)$  is a polynomial in $y$. More precisely from \eqref{def:Phi_0}, it is given explicitly by
\begin{align}
(\Psi_0)_{k,\ell} &=\sum_{h=0}^{2\ell} \binom{2\ell}{h}K_k(h)K_\ell(h)(-1)^{\frac{3h}{2}} y^{\frac{h}{2}}(1-y)^{\frac{2\ell-h}{2}} \nonumber\\
&=\sum_{h=0}^{\ell} (-1)^h\binom{2\ell}{2h}K_{2h}(k)K_{2h}(\ell)y^h(1-y)^{\ell-h} \nonumber\\
\label{eq:Psi0_k_y}
&=\sum_{h=0}^{\ell}\left[(-1)^h\binom{\ell}{h}\sum_{j=0}^{h}(-1)^j\binom{h}{j}K_{2j}(k)\right]y^h,\qquad 0\leq k \leq 2\ell,
\end{align}
Here we are using that $K_h(\ell)=0$ if $h$ is odd and $h\leq 2\ell-1$, the binomial theorem and the identities
$$
K_{2j}(\ell)=(-1)^j\binom{\ell}{j}\binom{2\ell}{2j}^{-1},\quad \binom{\ell}{j}\binom{\ell-j}{h-j}=\binom{\ell}{h}\binom{h}{j}.
$$
\begin{lem}
\label{lem:coef_psi}
We have
\begin{equation}
\label{eq:coef_ch}
\sum_{j=0}^{h}(-1)^j\binom{h}{j}K_{2j}(k)=\frac{(-k)_h(-2\ell+k)_h}{(-\ell)_h(-\ell+1/2)_h}.
\end{equation}
\end{lem}
\begin{proof}
First we rewrite the Krawtchouk polynomial $K_{2j}(k)=K_k(2j)$ as in \cite[Formula (1.10.1)]{KoekS} and invert the order of summation. We obtain
\begin{align}
\label{eq:coefficient_S}
\sum_{j=0}^{h}(-1)^j\binom{h}{j}K_{2j}(k) =\sum_{i=0}^k \frac{(-1)^i 2^i(-k)_i}{(-2\ell)_i} \, \sum_{j=0}^h (-1)^j\binom{h}{j}\binom{2j}{i}.
\end{align}
The inner sum is given explicitly by
$$\sum_{j=0}^h (-1)^j\binom{h}{j}\binom{2j}{i}=(-1)^h2^{2h-i}\binom{h}{i-h},$$
so that \eqref{eq:coefficient_S} becomes, using $(-k)_i/(-2\ell)_i=\binom{2\ell-i}{2\ell-i}\binom{2\ell}{k}^{-1}$, the following expression
$$
(-1)^h2^{2h}\binom{2\ell}{k}^{-1} \sum_{i=0}^k (-1)^i\binom{2\ell-i}{2\ell-i}\binom{h}{i-h}=(-1)^{h+k}2^{2h}\binom{2\ell}{k}^{-1}\binom{h+k-2\ell-1}{k-h}.
$$
The last sum can be evaluated explicitly using \cite[Formula (5.25)]{Gra}. Finally, a straightforward computation shows that the last expression is exactly the same as the one given on the right hand side of  \eqref{eq:coef_ch}.
\end{proof}

Now we construct a diagonal matrix $S(y)$ with the entries of the $\ell$-th column of $\Psi_0(y)$ as diagonal entries. Then we have
\begin{equation}\label{SS}
S(y)=\sum_{i=0}^{2\ell} \, (\Psi_0)_{i,\ell} \, E_{ii}.
\end{equation}
\begin{lem}
\label{lem:SS}
For all $k=0,\ldots, 2\ell$, we have
$$\frac{S(y)_{k+1,k+1}}{S(y)_{k,k}}\geq 0, \quad \text{for }y\in[0,1/2), \qquad \qquad \frac{S(y)_{k+1,k+1}}{S(y)_{k,k}}\leq 0, \quad \text{for }y\in(1/2,1].$$ 
\end{lem}
\begin{proof}
First we rewrite the entries of $S$ in the basis $\{(1-2y)^j\}$. It follows from \eqref{eq:Psi0_k_y} and Lemma \ref{lem:coef_psi} that
\begin{align*} 
S(y)_{k,k}=(\Psi_0)_{k,\ell} &=\sum_{j=0}^k \alpha_{k,j} (1-2y)^j,\qquad \alpha_{k,j}=\frac{(-1)^j}{j!}\sum_{s=0}^{k-j}  \frac{ 2^{-s-j} (-k)_{s+j} (-2\ell+k)_{s+j}}{s!(-\ell+1/2)_{s+j}}.
\end{align*}
Note that
\begin{align*}
\alpha_{k,j}&=\frac{(-1)^j 2^{-j} (-k)_j(-2\ell+k)_j}{j! (-\ell+1/2)_j} \sum_{s=0}^{k-j} \frac{(-k+j)_s (-2\ell+k+j)_s}{s!(-\ell+j+1/2)_s} \, 2^{-s},\displaybreak[0]\\
&=\frac{(-1)^j 2^{-j} (-k)_j(-2\ell+k)_j}{j! (-\ell+1/2)_j} \rFs{2}{1}{-k+j,-2\ell+k+j}{-\ell+j+1/2}{1/2},\\
&=\frac{(-1)^j 2^{-j} (-k)_j(-2\ell+k)_j}{j! (-\ell+1/2)_j} \frac{(k-j)!}{(-2\ell+2j)_{k-j}} \, C^{(-\ell+j)}_{k-j}(0)\\
&=\begin{cases}
0, &\mbox{if}\quad k-j \text{ is odd},\\
\D\frac{(-1)^j 2^{-j} (-k)_j(-2\ell+k)_j}{j! (-\ell+1/2)_j}\frac{(k-j)!}{(-2\ell+2j)_{k-j}} \frac{ (-1)^{(k-j)/2} (-\ell+j)_{(k-j)/2}}{((k-j)/2)!}, &\mbox{if}\quad k-j\text{ is even}.
\end{cases}\\
&=\begin{cases}
0, &\mbox{if}\quad k-j \text{ is odd},\\
\D\frac{k!(2\ell-k-j+1)_j(\ell-(k+j)/2+1)_{(k-j)/2}}{2^j j!((k-j)/2)!(\ell-j+1/2)_j}, &\mbox{if}\quad k-j\text{ is even}.
\end{cases}
\end{align*}
The third equality comes from the definition of Gegenbauer polynomials in terms of the hypergeometric function (see \cite[Formula (1.8.15)]{KoekS}), while the fourth equality comes from the value of the Gegenbauer polynomials at zero, see \cite[Table 18.6.1]{DLMF}. Observe that all coefficients $\alpha_{k,j}$ are nonnegative. Therefore we have

\begin{align*}
S(y)_{2k,2k}&=\sum_{j=0}^k\alpha_{2k,2j}(1-2y)^{2j},\quad k=0,1,\ldots,\ell,\\
S(y)_{2k+1,2k+1}&=(1-2y)\sum_{j=0}^{k}\alpha_{2k+1,2j+1}(1-2y)^{2j},\quad k=0,1,\ldots,\ell-1,
\end{align*}
from which the Lemma easily follows.
\end{proof}

\begin{prop}\label{Prop21}
Let $\Xi=S^{-1}\Omega^{(\nu)}S$, where $\Omega^{(\nu)}$ is given by \eqref{eqOmg}. Then we have
\begin{equation}\label{OpXi}
\Xi=y(1-y)\partial^2_y + A^{(\nu)}(y) \partial_y + Q^{(\nu)}(y),\quad \partial_y=\frac{d}{dy},
\end{equation}
where
\begin{align*}
A^{(\nu)}(y) &= 2y(1-y)S(y)^{-1}S'(y)+a^{(\nu)}(y), \\ 
Q^{(\nu)}(y) &= y(1-y)S(y)^{-1}S''(y) + a^{(\nu)}(y) S(y)^{-1}S'(y) + S(y)^{-1}F^{(\nu)}(y)S(y).
\end{align*}
Morover, the sum of the rows of $Q^{(\nu)}(y)-(\Lambda^{(\nu)}_0)_{\ell,\ell}$ and the off-diagonal terms of $Q^{(\nu)}$ are nonnegative for all $y\in[0,1]$.
\end{prop}
\begin{proof}
It follows from \eqref{eqOmg} that the spherical functions $\Psi^{(\nu)}_n$ are solutions of the differential equation
$$
y(1-y) \, \left[\Psi^{(\nu)}_n(y)\right]'' + a^{(\nu)}(y) \, \left[\Psi^{(\nu)}_n(y)\right]' + F^{(\nu)}(y)\, \Psi^{(\nu)}_n(y) = \Psi^{(\nu)}_n(y) \, \Lambda_n^{(\nu)},
$$
where $F^{(\nu)}(y)$ is defined by \eqref{eq:Fnu_SU2} and $\Lambda_n^{(\nu)}$ by \eqref{eigvalue}. A straightforward computation shows that the function $\chi_n=S^{-1}\Psi^{(\nu)}_n$ satisfies the following differential equation:
\begin{multline*}
y(1-y)\chi_n''(y) + (2y(1-y)S(y)^{-1}S'(y)+a^{(\nu)}(y))\chi_n'(y)  \\
+ [y(1-y)S(y)^{-1}S''(y) + a^{(\nu)}(y) S(y)^{-1}S'(y) + S(y)^{-1}F^{(\nu)}(y)S(y)]\chi_n(y)
= \chi_n(y) \Lambda_n^{(\nu)}.
\end{multline*}
This proves the first statement of the proposition. Observe that the fact that the sum of the rows of $Q^{(\nu)}(y)-(\Lambda_0^{(\nu)})_{\ell,\ell}$ is zero, is equivalent to 
$$
[y(1-y)S(y)^{-1}S''(y) + a^{(\nu)}(y) S(y)^{-1}S'(y) + S(y)^{-1}F^{(\nu)}(y)S(y)-(\Lambda_0^{(\nu)})_{\ell,\ell}]\bm e_{2\ell+1}=0,
$$
where $\bm e_{2\ell+1}^*=(1,1,\ldots,1)\in\mathbb{C}^{2\ell+1}$, which is, in turn, equivalent to
$$
[y(1-y)S''(y) + a^{(\nu)}(y) S'(y) + F^{(\nu)}(y)S(y)-(\Lambda_0^{(\nu)})_{\ell,\ell}]\bm e_{2\ell+1}=0.
$$
If we denote by  $(\Psi_0)_\ell$ the $\ell$-th column of $\Psi_0$, it follows from \eqref{SS} that
$$
y(1-y)(\Psi_0)''_\ell(y)  + a^{(\nu)}(y) (\Psi_0)'_{\ell}(y)  + F^{(\nu)}(y)(\Psi_0)_{\ell}(y)  = (\Lambda_0^{(\nu)})_{\ell,\ell},
$$
which is the $\ell$-th column of \eqref{eqOmg}. 

Finally, the off-diagonal terms of $Q^{(\nu)}$ come from the term $S(y)^{-1}F^{(\nu)}(y)S(y)$. More precisely we have
$$(S(y)^{-1}F^{(\nu)}(y)S(y))_{k,k+1}=\frac{i(2\ell+\nu-k)(1-2y)}{4y(1-y)} \, \frac{S(y)_{k+1,k+1}}{S(y)_{k,k}},$$
which is nonnegative for all $y\in[0,1]$ by Lemma \ref{lem:SS}. The proof for the $(k,k-1)$-th entry is analogous. This completes the proof of the proposition.
\end{proof}

\begin{rmk}
There are two properties of the matrix-valued function $Q^{(\nu)}$ which are essential in the forthcoming sections: first, the sum of the rows is equal  to zero and second, the off-diagonal terms are nonnegative for $y\in[0,1]$. 

It follows from the proof of Proposition \ref{Prop21} that, for the sum of the rows of $Q^{(\nu)}$ to be zero, the diagonal matrix $S$ can be replaced by any column of the function $\Psi^{(\nu)}_n(y)$, viewed as a diagonal matrix. Our specific choice of $S$ is due to the fact that it has a simple expression that allows us to verify the second property of $Q^{(\nu)}$.

The proof of the first property follows from a  general argument that can be extended in a straightforward way to any of the families of matrix-valued spherical functions associated to compact Gelfand Pairs studied in \cite{HvP, vPR1}. The main challenge in finding probabilistic interpretations for the new families is to find a suitable diagonal matrix $S$ so that the second property holds.
\end{rmk}

\subsection{Block reducibility of the weight matrix}
The commutant algebra of the weight $W^{(\nu)}(y)$, denoted by $Z^{(\nu)}=\{ T\in M_{2\ell+1}(\C)\mid  [T,W^{(\nu)}(y)]=0 \, \forall y\in[0,1]\},$ was computed in \cite[Proposition 2.6]{KdlRR} where it was shown that it is generated by the matrix $H$, where $H\in M_{2\ell+1}(\C)$ is the self-adjoint involution defined by $H\colon e_j \mapsto e_{2\ell-j}$. Therefore there is an orthogonal decomposition with respect to the $\pm 1$-eigenspaces of $H$. More precisely, let $Y$ defined by
\begin{equation}
\label{YY}
\begin{split}
Y&=\frac{1}{\sqrt{2}}\begin{pmatrix} I_{\ell+\frac12} & H_{\ell+\frac12} \\
-H_{\ell+\frac12} & I_{\ell+\frac12} \end{pmatrix},\text{ if }\ell
=\frac{2n+1}{2},\quad n\in \mathbb{N},\\
Y&=\frac{1}{\sqrt{2}}\begin{pmatrix} I_{\ell}& 0 & H_{\ell} \\ 0 & \sqrt{2} & 0 \\
-H_{\ell}& 0 & I_{\ell} \end{pmatrix},\text{ if }\ell\in \mathbb{N}.
\end{split}
\end{equation}
Then
$$
\widetilde W(y)=YW^{(\nu)}(y)Y^*=\left(
\begin{array}{c|c}
W_1(y) & 0 \\
\hline
0& W_2(y) 
\end{array}
\right),
$$
where $W_1(y)$ is a $(\ell+1)\times(\ell+1)$ weight matrix and $W_2(y)$ is a $\ell\times\ell$ weight matrix. Observe that by  \cite[Example 4.2]{KRo} no further non-orthogonal decomposition is possible. We will use this block matrix decomposition in the next section to analyze two independent processes generated by the weight matrices $W_1(y)$ and $W_2(y)$. As we will see the probabilistic interpretation of these examples will not change under this transformation.

\section{The $\ell=1$ case}\label{SEC3}

For the $\ell=1$ case, the matrix $Y$ in \eqref{YY} is given by
\begin{equation}\label{YY3}
Y=\frac{1}{\sqrt{2}}\begin{pmatrix} 1&0&1 \\ 0&\sqrt{2}&0\\-1 &0&1 \end{pmatrix}.
\end{equation}
Therefore
$$
\widetilde{W}(y)=YW^{(\nu)}(y)Y^*=\left(
\begin{array}{c|c}
W_1(y) & \begin{array}{c} 0\\0\end{array} \\
\hline
\begin{array}{cc} 0&0\end{array} & w_2(y) 
\end{array}
\right),\quad y\in[0,1],
$$
where $W_1(y)$ is a $2\times2$ weight matrix and $w_2(y)$ is a positive scalar weight. This matrix $Y$ is unique up to linear polynomial combinations of $Y$. For this case, in order to study conveniently the stochastic processes behind, it will be appropriate to take different matrix transformations.

\subsection{Two birth-and-death models}

We take in this case the transformation matrix $T$ given by
$$
T=I_3+Y^2=\begin{pmatrix} 1&0&1 \\ 0&2&0\\-1 &0&1 \end{pmatrix},
$$
where $Y$ is given by \eqref{YY3}. Consider the monic matrix-valued orthogonal polynomials $P_n^{(\nu)}(y)$ corresponding to the weight matrix $W^{(\nu)}(y)$ defined in \eqref{WWnu}. With this transformation we have
$$
\widetilde{W}(y)=TW^{(\nu)}(y)T^*=\left(
\begin{array}{c|c}
W_1(y) & \begin{array}{c} 0\\0\end{array} \\
\hline
\begin{array}{cc} 0&0\end{array} & w_2(y) 
\end{array}
\right),\quad y\in[0,1],
$$
and 
$$
\widehat{P}_n(y)=TP_n^{(\nu)}(y)T^{-1}=\left(
\begin{array}{c|c}
P_{n,1}(y) & \begin{array}{c} 0\\0\end{array} \\
\hline
\begin{array}{cc} 0&0\end{array} & p_{n,2}(y) 
\end{array}
\right),
$$
where $\widehat{P}_n(y)$ is again a monic family. We normalize this family conveniently choosing a sequence of diagonal matrices $L_n$ such that $Q_n(y)=L_n\widehat{P}_n(y)$ satisfies
\begin{equation}\label{Qn0}
Q_n(0)\bm{e}_3=\bm{e}_3,
\end{equation}
where $\bm{e}_N$ denotes the column vector of dimension $N$ of all components equal to 1, i.e. $\bm{e}_N=(1,1,\ldots,1)^*$. This sequence of diagonal matrices is given by
$$
L_{2n}=4^{n}\begin{pmatrix} \frac{(\nu+n+1)_{n+1}}{(1+\nu)(\nu+3/2)_n}&0&0 \\ 0&\frac{(\nu+2n)(\nu+n+1)_{n}}{\nu(\nu+3/2)_n}&0\\0 &0&\frac{(\nu+n+1)_{n}}{(\nu+3/2)_n} \end{pmatrix},\quad n\geq0,
$$
and
$$
L_{2n+1}=-2\cdot4^{n}\begin{pmatrix} \frac{(\nu+n+2)_{n+1}}{(1+\nu)(\nu+3/2)_n}&0&0 \\ 0&\frac{(\nu+2n+1)(\nu+n+2)_{n}}{\nu(\nu+3/2)_n}&0\\0 &0&\frac{(\nu+n+2)_{n}}{(\nu+3/2)_n} \end{pmatrix},\quad n\geq0.
$$
$Q_n$ can also be divided by blocks
\begin{equation}\label{QQ2}
Q_n(y)=\left(
\begin{array}{c|c}
Q_{n,1}(y) & \begin{array}{c} 0\\0\end{array} \\
\hline
\begin{array}{cc} 0&0\end{array} & q_{n,2}(y) 
\end{array}
\right).
\end{equation}
Observe also that the norms of $Q_n$ with respect to $\widetilde W$ are related with the norms of $P_n^{(\nu)}$ with respect to $W^{(\nu)}$ as follows
\begin{equation}\label{RelNor}
\|Q_n\|^2_{\widetilde W}=L_nT\|P_n^{(\nu)}\|^2_{W^{(\nu)}}(L_nT)^*,\quad n\geq0,
\end{equation}
where $\|P_n^{(\nu)}\|^2_{W^{(\nu)}}$ are given by \eqref{NormsG}.

From \eqref{TTRRPn} we see that the sequence of matrix-valued orthogonal polynomials $Q_n(y)$ satisfies a three-term recurrence relation of the form
\begin{equation}\label{TTRRQ}
-yQ_n(y)=A_nQ_{n+1}(y)+B_nQ_n(y)+C_nQ_{n-1}(y),
\end{equation}
where the coefficients are given by
\begin{align*}
A_n&=-L_nL_{n+1}^{-1}=\left(
\begin{array}{cc|c}
\frac{2\nu+n+2}{4(\nu+n+2)}&0&0\\
0&\frac{(n+\nu)(2\nu+n+2)}{4(\nu+n+1)^2} &0\\
\hline
0&0& \frac{2\nu+n+2}{4(\nu+n+1)}
\end{array}
\right),\quad n\geq0,\\
B_n&=-L_nTB_n^{(\nu)}(L_nT)^{-1}=\left(
\begin{array}{cc|c}
-\frac{1}{2}&\frac{\nu}{2(\nu+n)(\nu+n+2)}&0\\
\frac{1+\nu}{2(\nu+n+1)^2} &-\frac{1}{2} &0\\
\hline
0&0& -\frac{1}{2}
\end{array}
\right),\quad n\geq0,\\
C_n&=-L_nTC_n^{(\nu)}(L_{n-1}T)^{-1}=\left(
\begin{array}{cc|c}
\frac{n}{4(\nu+n)}&0&0\\
0&\frac{n(\nu+n+2)}{4(\nu+n+1)^2} &0\\
\hline
0&0& \frac{n}{4(\nu+n+1)}
\end{array}
\right),\quad n\geq1.
\end{align*}
The corresponding Jacobi matrix is a block tridiagonal matrix with the property that the diagonal entries are negative, the off-diagonal entries are nonnegative and the sum of each row equals 0 (as a consequence of \eqref{Qn0} and \eqref{TTRRQ}). Therefore the Jacobi matrix is the matrix of an infinitesimal operator associated with a continuous-time quasi-birth-and-death process with two-dimensional state space $\N\times\{1,2,3\}$. As we can see from the division by blocks of the coefficients $A_n,B_n,C_n$, this process splits into two independent processes. The first one is a continuous-time quasi-birth-and-death process with two-dimensional state space $\N\times\{1,2\}$ with coefficients
\begin{align}
\nonumber A_{n,1}&=\left(
\begin{array}{cc}
\frac{2\nu+n+2}{4(\nu+n+2)}&0\\
0&\frac{(n+\nu)(2\nu+n+2)}{4(\nu+n+1)^2}
\end{array}
\right),\quad n\geq0,\\
\label{coeff2}B_{n,1}&=\left(
\begin{array}{cc}
-\frac{1}{2}&\frac{\nu}{2(\nu+n)(\nu+n+2)}\\
\frac{1+\nu}{2(\nu+n+1)^2} &-\frac{1}{2}
\end{array}
\right),\quad n\geq0,\\
\nonumber C_{n,1}&=\left(
\begin{array}{cc}
\frac{n}{4(\nu+n)}&0\\
0&\frac{n(\nu+n+2)}{4(\nu+n+1)^2} 
\end{array}
\right),\quad n\geq1.
\end{align}
Therefore, the matrix of the infinitesimal operator (conservative) is a pentadiagonal matrix given by
\begin{equation}\label{AA1}
\mathcal{A}_1=\left(
\begin{array}{cc|cc|cc|cc|c}
-\frac{1}{2}&\frac{1}{2(\nu+2)}&\frac{\nu+1}{2(\nu+2)}&0&0&0&0&0&\cdots\\
\frac{1}{2(\nu+1)} &-\frac{1}{2} &0&\frac{\nu}{2(\nu+1)}&0&0&0&0&\cdots\\
\hline
\frac{1}{4(\nu+1)}&0&-\frac{1}{2}&\frac{\nu}{2(\nu+1)(\nu+3)}&\frac{2\nu+3}{4(\nu+3)}&0&0&0&\cdots\\
0&\frac{\nu+3}{4(\nu+2)^2}&\frac{1+\nu}{2(\nu+2)^2} &-\frac{1}{2} &0&\frac{(1+\nu)(2\nu+3)}{4(\nu+2)^2}&0&0&\cdots\\
\hline
0&0&\frac{1}{2(\nu+2)}&0&-\frac{1}{2}&\frac{\nu}{2(\nu+2)(\nu+4)}&\frac{\nu+2}{2(\nu+4)}&0&\cdots\\
0&0&0&\frac{\nu+4}{2(\nu+3)^2}&\frac{1+\nu}{2(\nu+3)^2} &-\frac{1}{2} &0&\frac{(2+\nu)^2}{2(\nu+3)^2}&\cdots\\
\hline
\vdots&\vdots&\vdots&\vdots&\ddots&\ddots&\ddots&\ddots&\ddots
\end{array}
\right)
\end{equation}
The second process is a regular birth-and-death process with rational birth and death parameters given by
\begin{equation}\label{bdp}
\lambda_n=\frac{2\nu+n+2}{4(\nu+n+1)},\quad \mu_n=\frac{n}{4(\nu+n+1)},\quad n\geq0.
\end{equation}
Therefore, the matrix of the infinitesimal operator (again conservative) is a tridiagonal matrix given by
\begin{equation}\label{AA2}
\mathcal{A}_2=\begin{pmatrix}
-\frac{1}{2}&\frac{1}{2}&0&&\\
\frac{1}{4(\nu+2)}&-\frac{1}{2}&\frac{2\nu+3}{4(\nu+2)}&0&\\
0&\frac{1}{2(\nu+3)}&-\frac{1}{2}&\frac{\nu+2}{2(\nu+3)}&0\\
&&\ddots&\ddots&\ddots
\end{pmatrix}.
\end{equation}

The good thing about these two processes is that we have explicitly all the elements to perform the spectral analysis (the weights, orthogonal polynomials and norms), so we can have a Karlin-McGregor formula for the transition probabilities of both processes, which is unique since all coefficients are bounded (see \cite[Section 4.3]{BW}). Let us study the probabilistic properties of each one of these processes.

%
%
\smallskip

\noindent {\bf (1)} Let $\{X_t: t\geq0\}$ be the birth-and-death process associated with the infinitesimal operator \eqref{AA2}. The transition probabilities are given by
$$
P_{ij}^{(2)}(t)=\mathbb{P}(X_t=j | X_0=i).
$$
The potential coefficients can be explicitly calculated from the definition of $\lambda_n$ and $\mu_n$ in \eqref{bdp}. Indeed,
$$
\pi_0=1,\quad \pi_n=\frac{2(\nu+n+1)(2\nu+3)_{n-1}}{n!},\quad n\geq1.
$$
The scalar weight is given by
\begin{equation}\label{ww1}
w_2(y)=\frac{4^{\nu+1}(\nu+1)_2}{\nu+1/2}\left[y(1-y)\right]^{\nu+1/2},\quad y\in[0,1],\quad \nu>-3/2.
\end{equation}
The polynomials $q_{n,2}(y)$ in \eqref{QQ2} are a special instance of the Gegenbauer polynomials on $[0,1]$ with the property that $q_{n,2}(0)=1$. In particular, we have that 
$$
\pi_n=\frac{\|q_{0,2}\|^2_{w_2}}{\|q_{n,2}\|^{2}_{w_2}},\quad \|q_{0,2}\|^2_{w_2}=\frac{\sqrt{\pi}(\nu+2)\Gamma(\nu+1/2)}{\Gamma(\nu+1)}.
$$
We can therefore perform the spectral analysis of the process and have the Karlin-McGregor representation
\begin{align*}
    P_{ij}^{(2)}(t)&=\frac{1}{\|q_{j,2}\|^2_{w_2}}\int_0^1 e^{-yt}q_{i,2}(y)q_{j,2}(y)w_2(y)dy\\
    &=\frac{2(\nu+j+1)(2\nu+3)_{j-1}4^{\nu+1}\Gamma(\nu+2)}{j!\sqrt{\pi}\Gamma(\nu+3/2)}\int_0^1 e^{-yt}q_{i,2}(y)q_{j,2}(y)\left[y(1-y)\right]^{\nu+1/2}dy.
\end{align*}
We can also analyze the recurrence of the process in terms of the weight $w_2(y)$. Indeed, a necessary and sufficient condition in order for the process to be recurrent is that 
$$
\int_0^1\frac{w_2(y)}{y}dy=\infty.
$$
From the definition \eqref{ww1} we see that this is possible only when $-3/2<\nu\leq-1/2$. Otherwise (if $\nu>-1/2$) the process will be transient. For the values where the process is recurrent it is possible to see that $\sum \pi_n=\infty$, so the process will be \emph{null recurrent} and it can never be positive recurrent or ergodic. This behavior can be seen in Figure \ref{ctmc1}. In the first plot, we fix $\nu=-5/4$ (recurrent), so the trajectories can reach the boundary state 0 recurrently. In the second plot $\nu=0$ (transient) so the length of the queue tends to go to infinity and never comes back.

\begin{figure}[h]
\begin{center}
\vspace{-0.4cm}
\includegraphics[height=8.4cm]{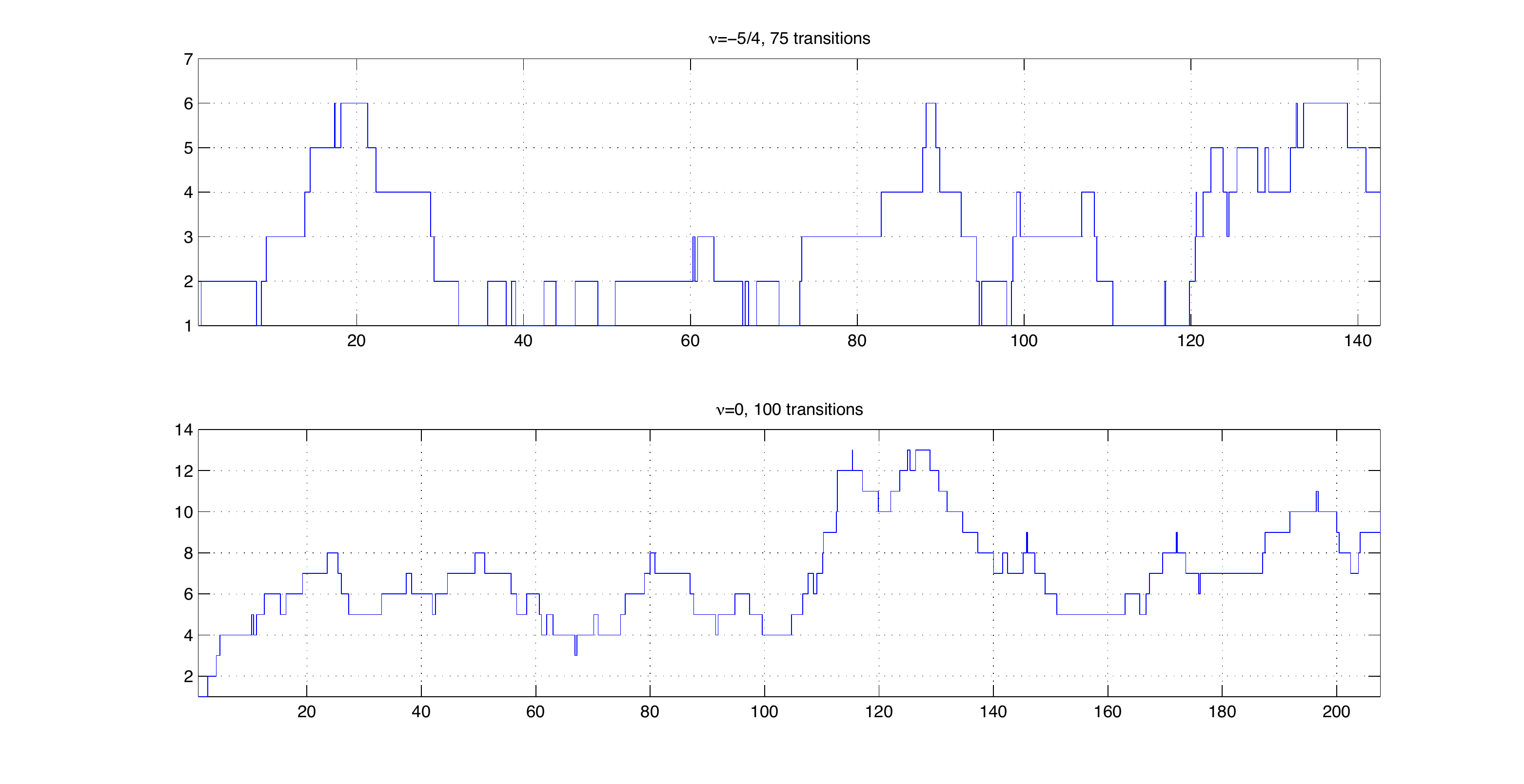}
\vspace{-1.2cm}
\end{center}
\caption{Trajectories of the queue starting at $X_0=1$ (the state space is $\{1,2,3,\ldots\}$) given by coefficients \eqref{bdp} for the value of the parameters $\nu=-5/4$ (null recurrent) and $\nu=0$ (transient).}
\label{ctmc1}
\end{figure}

This birth-and-death process can be seen as a rational variant of the one-server queue as the length of the queue increases. As $n\to\infty$ we see that both birth and death coefficients in \eqref{bdp} converges to $1/4$. These coefficients make a difference when the length of the queue is short depending on the parameter $\nu$ (except when $\nu=-1$ where both coefficients are constant). But when it is growing the queue behaves like the one-server queue.

\smallskip

\noindent {\bf (2)} Let $\{Z_t=(X_t,Y_t): t\geq0\}$ be the two-dimensional quasi-birth-and-death process associated with the infinitesimal operator \eqref{AA1}. The transition probabilities are given by
$$
\left(P_{ij}^{(1)}(t)\right)_{i'j'}=\mathbb{P}(X_t=j, Y_t=j' | X_0=i, Y_0=i'),\quad i,j\in\N,\quad i',j'\in\{1,2\}.
$$
Observe that $P^{(1)}(t)$ is a block matrix. The probability of going from state $(i,i')$ to state $(j,j')$ in time $t$ is given by the element in the position $(i',j')$ of the matrix $P_{ij}^{(1)}(t)$. The weight matrix is supported on $[0,1]$ and is given by
\begin{equation}\label{WW1}
W_1(y)=4^{\nu+1/2}(\nu+2)\left[y(1-y)\right]^{\nu-1/2} \begin{pmatrix}1-\frac{2(1+\nu)}{\nu+1/2}y(1-y) &1-2y\\
1-2y&1-\frac{2\nu}{\nu+1/2}y(1-y)
\end{pmatrix},
\end{equation}
where now, in order that the infinitesimal matrix \eqref{AA1} has a probabilistic interpretation, we need to impose $\nu\geq0$ (although the weight matrix is well defined for $\nu>-1/2$). Each block entry $(i,j)$ of $P^{(1)}(t)$ admits a Karlin-McGregor integral representation of the form (see \cite{DR})
\begin{equation*}\label{KMcG1}
    P_{ij}^{(1)}(t)=\bigg(\int_0^1 e^{-yt}Q_{i,1}(y)W_1(y)Q_{j,1}^*(y)dx\bigg)\bigg(\int_0^1
    Q_{j,1}(y)W_1(y)Q_{j,1}^*(y)dy\bigg)^{-1}.
\end{equation*}
As it was shown in \cite{dI1} the inverse matrix of the norms of the polynomials $Q_{n,1}$ in \eqref{QQ2} are exactly the \emph{matrix-valued potential coefficients}, defined by
$$
\Pi_0=\|Q_{0,1}\|_{W_1}^{-2},\quad \Pi_n=\left(\|Q_{n,1}\|_{W_1}^2\right)^{-1}=(C_{1,1}^*C_{2,1}^*\cdots C_{n,1}^*)^{-1}\|Q_{0,1}\|_{W_1}^{-2}A_{0,1}A_{1,1}\cdots A_{n-1,1},
$$
where $A_{n,1}$ and $C_{n,1}$ are defined in \eqref{coeff2}. Since $A_{n,1}$ and $C_{n,1}$ are diagonal matrices and the norm of $Q_{0,1}=I_2$ is given by
$$
\|Q_{0,1}\|_{W_1}^{-2}=\frac{\Gamma(\nu+1)}{\sqrt{\pi}(\nu+2)\Gamma(\nu+1/2)}\begin{pmatrix}1&0\\0&\D\frac{\nu+1}{\nu+2}\end{pmatrix},
$$
we can calculate an explicit expression of the matrix-valued potential coefficients with the help of \eqref{NormsG} and \eqref{RelNor}, given by
$$
\Pi_0=\|Q_{0,1}\|_{W_1}^{-2},\quad \Pi_n=\frac{2\Gamma(\nu+2)(2\nu+3)_{n-1}}{\sqrt{\pi}n!(\nu+2)\Gamma(\nu+1/2)}\begin{pmatrix} \frac{\nu+1}{\nu+n+1}&0\\0&\frac{\nu(\nu+n+1)}{(\nu+n)(\nu+n+2)}\end{pmatrix},\quad n\geq1.
$$
Not only that, but according to Theorem 3.1 of \cite{dI1} we can compute \emph{explicitly} the invariant measure of the process, given by
\begin{align*}
\nonumber\bm\pi&=\left((\Pi_0\bm e_2)^*;(\Pi_1\bm e_2)^*(\Pi_2\bm e_2)^*;\cdots\right),\quad\bm e_2^*=(1,1),\\
&=\frac{\Gamma(\nu+1)}{\sqrt{\pi}\Gamma(\nu+1/2)(\nu+2)}\left(1,\frac{\nu+1}{\nu+2};\frac{2(\nu+1)^2}{\nu+2},\frac{2\nu(\nu+2)}{\nu+3};\cdots\right).
\end{align*}
We observe that for all values of $\nu$
\begin{align*}
\sum_{n=0}^{\infty}\bm\pi_n&=\frac{\Gamma(\nu+1)}{\sqrt{\pi}\Gamma(\nu+1/2)(\nu+2)}
\left[1+2(\nu+1)^2\sum_{n=1}^{\infty}\frac{(2\nu+3)_{n-1}}{n!(\nu+n+1)}\right.\\
&\left.\frac{\nu+1}{\nu+2}\left(1+2\nu(\nu+2)\sum_{n=1}^{\infty}\frac{(2\nu+3)_{n-1}(\nu+n+1)}{n!(\nu+n)(\nu+n+2)}\right)\right]=\infty.
\end{align*}

We can analyze the recurrence of the process in terms of the weight $W_1(y)$. According to Theorem 4.1 in \cite{DR} the process is $\alpha$-recurrent if and only if for some $l=1,2,$ we have that
$$
e_l^*\left(\int\frac{W_1(y)}{x-\alpha}dy\right)e_l=\infty,
$$
where $e_1^*=(1,0)$ and $e_2^*=(0,1)$. Since in our case the process is irreducible and the weight matrix is supported in the interval $[0,1]$, then $\alpha=0$, in which case $\alpha$-recurrence is equivalent to regular recurrence. From the definition \eqref{WW1} we see that the process is recurrent only when $0\leq\nu\leq 1/2$. Otherwise (if $\nu>1/2$) the process will be transient. For the values where the process is recurrent we have that $\sum\bm\pi_n=\infty$, so the process will always be \emph{null recurrent}. This behavior can be seen in Figure \ref{ctmc2} and it is similar to the previous example.

\begin{figure}[h]
\begin{center}
\vspace{-0.5cm}
\includegraphics[height=8.6cm]{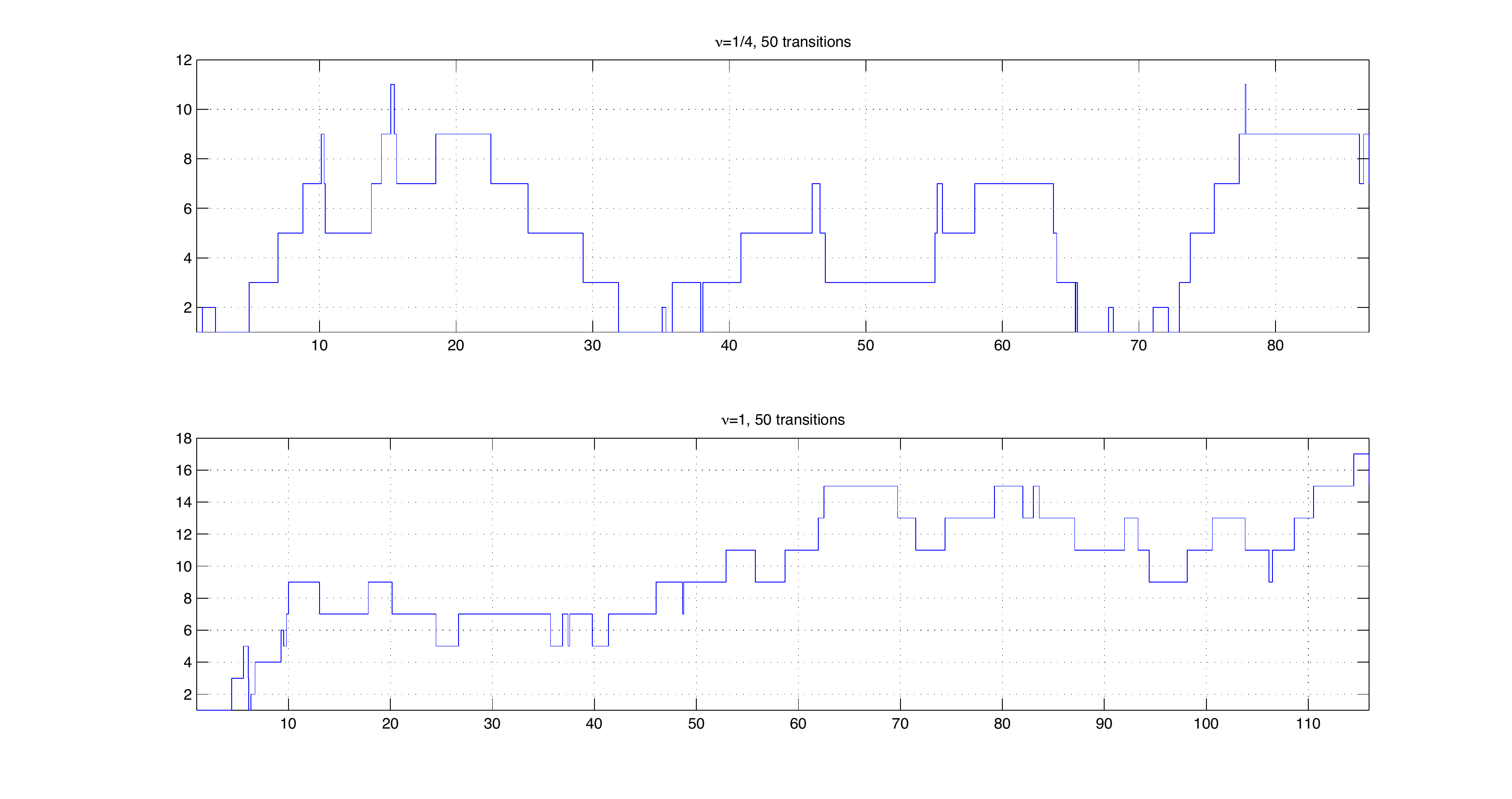}
\vspace{-1.0cm}
\end{center}
\caption{Trajectories of the queue starting at $X_0=1$ and $Y_0=1$ (the state space is $\{1,2,3,\ldots\}\times\{1,2\}$) with infinitesimal operator \eqref{AA1} for the value of the parameters $\nu=1/4$ (null recurrent) and $\nu=1$ (transient).}
\label{ctmc2}
\vspace{-.2cm}
\end{figure}

This quasi-birth-and-death process (with 2 phases) may be viewed as a queue with state space $\{0,1,\ldots\}$ and the following behavior. There are two ways of increasing or decreasing the length of the queue, either by 1 element or by 2. If the process moves along any of the phases, then the process can add (or remove) 2 elements to the queue. On the contrary, if the process moves from one phase to another, then the process add (or remove) 1 element to the queue. The transitions of phases are ruled by entries $(1,2)$ and $(2,1)$ of $B_{n,1}$ in \eqref{coeff2}. As $n\to\infty$ these coefficients tend to 0, meaning that as the length of the queue increases, it is very unlikely that a transition between phases occurs. This behavior can be seen more closely in Figure \ref{ctmc22}. As $n\to\infty$ the birth and death rates for each phase tend to $1/4$, so it behaves like the one-server queue \emph{but} adding or removing 2 elements to the queue. Therefore this quasi-birth-and-death process may be viewed as a rational variation of a couple of  one-server queues where the interaction between them is remarkable in the first states of the queue.

\begin{figure}[h]
\begin{center}
\vspace{-0.3cm}
\includegraphics[height=6.4cm]{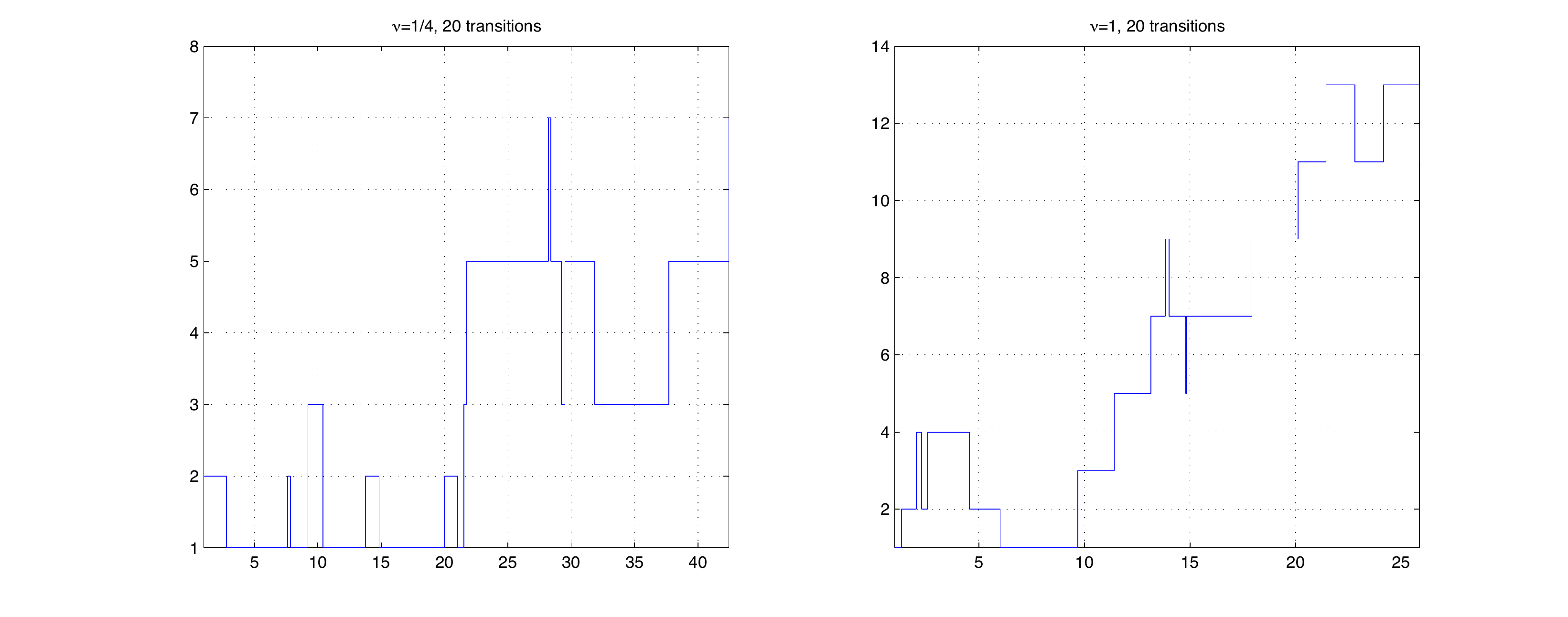}
\vspace{-1cm}
\end{center}
\caption{Trajectories of the queue starting at $X_0=1$ and $Y_0=1$ (the state space is $\{1,2,3,\ldots\}\times\{1,2\}$) with infinitesimal operator \eqref{AA1} for the value of the parameters $\nu=1/4$ and $\nu=1$. The possibilities of increasing or decreasing by 1 the queue are higher when the length of the queue is shorter.}
\label{ctmc22}
\end{figure}

The importance about this example, as far as the authors know, is that it is the first nontrivial \emph{continuous-time} level-dependent quasi-birth-and-death process where a complete spectral analysis can be given.


\subsection{Two diffusion models}\label{S32}
In this case we have to follow the conjugation given by the matrix $S(y)$ in \eqref{SS} (see also \eqref{eq:Psi0_k_y} and \eqref{eq:coef_ch}), which it is given by
$$
S(y)=\begin{pmatrix} 1 & 0& 0\\0&1-2y&0\\0&0&1\end{pmatrix}.
$$
Additionally, we consider the transformation matrix $T$ given by
\begin{equation}\label{TTT}
T=-\frac{\sqrt{2}}{2}I_3+(1+\sqrt{2})Y-\frac{\sqrt{2}}{2}Y^2=\begin{pmatrix} 1&0&1 \\ 0&1&0\\-1 &0&1 \end{pmatrix},
\end{equation}
where $Y$ is given by \eqref{YY3}. These two transformations allow us to derive second-order differential operators with stochastic interpretation, according to Proposition \ref{Prop21}, as well as splitting the weights and polynomials into blocks, which will not change the probabilistic interpretation of these operators.

Let $\{W^{(\nu)},D^{(\nu)}\}$ be the pair given by \eqref{WWnu} and \eqref{diffop}, respectively. We consider a transformation of this pair according to the following function
$$
R(y)=\Psi_0^{-1}(y)S(y)T^*,
$$
where $\Psi_0(y)$ is given by \eqref{def:Phi_0}. The new pair is $\{\widetilde{W},\widetilde{D}\}$, where
$$
\widetilde{W}(y)=R^*(y)W^{(\nu)}(y)R(y),\quad \widetilde{D} F(y)=R^{-1}(y)D^{(\nu)}\left[(R(y)F(y)\right].
$$
Observe that $\widetilde{D} $ is the operator $(T^*)^{-1}\Xi T^*$, where $\Xi$ is given in Proposition \ref{Prop21}. Consider now $P_n^{(\nu)}$ the monic family of matrix-valued orthogonal polynomials with respect to $W^{(\nu)}$ given by \eqref{defmonic} such that $D^{(\nu)}(P_n^{(\nu)})^*=(P_n^{(\nu)})^*\Lambda_n^{(\nu)}$, where $\Lambda_n^{(\nu)}$ is given by \eqref{eigvalue}. Define the sequence of matrix-valued functions
\begin{equation}\label{eq:Qn_l=1_diffussion}
Q_n(y)=R^{-1}(y)(P_n^{(\nu)}(y))^*T^*.
\end{equation}
Observe that $Q_n$ are no longer real matrix-valued polynomials since
$$
R^{-1}(y)=\begin{pmatrix}
1-2y&1&1-2y\\
1/(1-2y)&1&1/(1-2y)\\
2i\sqrt{y(1-y)}&0&-2i\sqrt{y(1-y)}
\end{pmatrix}.
$$
Then it is easy to see that $Q_n$ is a family of matrix-valued orthogonal functions with respect to $\widetilde{W}$, which is  given by
\begin{equation*}
\widetilde{W}(y)=\frac{4^{\nu-1}(2+\nu)[y(1-y)]^{\nu-1/2}}{\nu+1/2}\left(
\begin{array}{cc|c}
1+\nu&0&0\\
0&\nu(1-2y)^2&0\\
\hline
0&0&1+\nu
\end{array}
\right).
\end{equation*}
$Q_n$ can also be divided by blocks
\begin{equation}\label{QQ}
Q_n(y)=\left(
\begin{array}{c|c}
Q_{n,1}(y) & \begin{array}{c} 0\\0\end{array} \\
\hline
\begin{array}{cc} 0&0\end{array} & q_{n,2}(y) 
\end{array}
\right),
\end{equation}
and the norms are given by
\begin{equation}\label{RelNor2}
\|Q_n^*\|^2_{\widetilde W}=T\|P_n^{(\nu)}\|^2_{W^{(\nu)}}T^*.
\end{equation}
Additionally, $Q_n$ is eigenfunction of the second-order differential operator
\begin{align*}
\widetilde{D}=y(1-y)\partial_y^2+&\left(
\begin{array}{cc|c}
(\nu+1/2)(1-2y)&0&0\\
0&(\nu+3/2)(1-2y)-\D\frac{1}{1-2y} &0\\
\hline
0&0& (\nu+1/2)(1-2y)
\end{array}
\right)\partial_y\\
\nonumber&+\frac{1}{2y(1-y)}\left(
\begin{array}{cc|c}
-\nu(1-2y)^2&\nu(1-2y)^2&0\\
1+\nu&-(1+\nu) &0\\
\hline
0&0& -\nu(1-2y)^2
\end{array}
\right),
\end{align*}
i.e. $\widetilde{D}Q_n=Q_n\widetilde{\Lambda}_n$, where the eigenvalue is $\widetilde{\Lambda}_n=\Lambda_n^{(\nu)}+\nu^2+2\nu-4$ and in this case it is given by
\begin{equation*}\label{Eig3}
\widetilde{\Lambda}_n=\left(
\begin{array}{cc|c}
-1-n(n+2\nu+2)&0&0\\
0&-n(n+2\nu+2)&0\\
\hline
0&0& -1-n(n+2\nu+2)
\end{array}
\right),\quad n\geq0.
\end{equation*}
This second-order differential operator can be identified with the infinitesimal operator of a two-dimensional diffusion process (also known as switching diffusion processes) with state space $[0,1]\times\{1,2,3\}$. As before, the division by blocks gives two independent processes. The first one is a switching diffusion process with state space $[0,1]\times\{1,2\}$ with infinitesimal operator given by
\begin{align}\label{DifOpB1}
D_1=y(1-y)\partial_y^2+&\left(
\begin{array}{cc}
(\nu+1/2)(1-2y)&0\\
0&(\nu+3/2)(1-2y)-\D\frac{1}{1-2y}
\end{array}
\right)\partial_y\\
\nonumber&+\frac{1}{2y(1-y)}\left(
\begin{array}{cc}
-\nu(1-2y)^2&\nu(1-2y)^2\\
1+\nu&-(1+\nu)
\end{array}
\right),\quad \nu\geq0,
\end{align}
with eigenvalue
\begin{equation}\label{Eig31}
\Lambda_{n,1}=\left(
\begin{array}{cc}
-1-n(n+2\nu+2)&0\\
0&-n(n+2\nu+2)
\end{array}
\right),\quad n\geq0,
\end{equation}
and weight matrix
\begin{equation}\label{WWq}
W_1(y)=\frac{4^{\nu-1}(2+\nu)[y(1-y)]^{\nu-1/2}}{\nu+1/2}\left(
\begin{array}{cc}
1+\nu&0\\
0&\nu(1-2y)^2
\end{array}
\right).
\end{equation}
Observe that the independent coefficient of $D_1$ (depending on $y$) is the matrix of the infinitesimal operator of a continuous-time birth-and-death process with two states.

The second process is a regular diffusion process with a \emph{killing factor}, which infinitesimal operator is given by
\begin{equation}\label{DifOpB2}
D_2=y(1-y)\partial_y^2+(\nu+1/2)(1-2y)\partial_y-\D\frac{\nu(1-2y)^2}{2y(1-y)},\quad \nu\geq0,
\end{equation}
with eigenvalue
\begin{equation}\label{Eig32}
\lambda_{n,2}=-1-n(n+2\nu+2),
\end{equation}
and weight function
\begin{equation}\label{ww2}
w_2(y)=\frac{4^{\nu-1}(1+\nu)_2[y(1-y)]^{\nu-1/2}}{\nu+1/2}.
\end{equation}
Observe that the independent coefficient of $D_2$ (depending on $y$) is never positive, so it is the killing factor of a diffusion process.

We can perform again the spectral analysis of these two operators since we have an explicit expression of the weights, orthogonal functions and norms. Let us study the probabilistic properties of each one of these diffusion processes.

\medskip

\noindent {\bf (1)} Let $\{X_t, t\geq0\}$ be the diffusion process with killing associated with the infinitesimal operator \eqref{DifOpB2} and call $p(t;x,dy)$ the probability transition distribution of the process if it has not been killed yet. It is well known that $p(t;x,dy)$ has a density $p(t;x,y)$ and it is given by (see for instance Section 15.13 of \cite{KT2})
$$
p(t;x,y)=\sum_{n=0}^{\infty}e^{\lambda_{n,2}t}q_{n,2}(x)\overline{q_{n,2}(y)}\pi_nw_2(y),
$$
where $w_2(y)$ is given by \eqref{ww2}, the eigenvalue $\lambda_{n,2}$ is given by \eqref{Eig32} and $\pi_n$ are the inverse of the squared norms of the functions $q_{n,2}$ in \eqref{QQ}. The family  of functions $q_{n,2}$ can be written in the following way
$$
q_{n,2}(y)=-\frac{i\, n!\, \sqrt{y(1-y)}}{2^{n-2}(\nu+1)_n}  C_n^{(\nu+1)}(y),
$$
where $C_n^{(\lambda)}$ is the family of Gegenbauer polynomials, see \cite[(1.8.15)]{KoekS}. The norms with respect to \eqref{ww2} follows from the explicit expression \eqref{NormsG}, \eqref{eq:Qn_l=1_diffussion} and \eqref{RelNor2}:
\begin{equation}\label{norm2q}
\pi_n^{-1}=\|q_{n,2}^*\|_{w_2}^2=\frac{\pi n!(n+\nu+1)(\nu+1)_2\Gamma(n+2\nu+2)}{16^{n} 4^{\nu}(2\nu+1) \Gamma(n+\nu+2)^2}.
\end{equation} 
Therefore $p(t;x,y)$ can be written in the following way
\begin{align*}
p(t;x,y)&=e^{-t}\sqrt{x(1-x)}\frac{4^{\nu+1}(1+\nu)_2[y(1-y)]^{\nu}}{\nu+1/2}\sum_{n=0}^{\infty}\frac{e^{-n(n+2\nu+2)t} \, (n!)^2 \, \pi_n}{4^n \, (\nu+1)_n^2}\, C_n^{(\nu+1)}(x)C_n^{(\nu+1)}(y)\\
&=\frac{2}{\pi}e^{-t}\sqrt{x(1-x)}4^{2\nu+1}\Gamma(\nu+1)^2[y(1-y)]^{\nu}\times\\
&\hspace{3cm}\times\sum_{n=0}^{\infty}\frac{e^{-n(n+2\nu+2)t} \, n! \, 4^n(n+\nu+1)}{\Gamma(n+2\nu+1)}\, C_n^{(\nu+1)}(x)C_n^{(\nu+1)}(y).
\end{align*}
It is well known that the killing time $\xi$ is a random variable distributed according the law
$$
\mathbb{P}\left(\xi> t | \{X_s, s\geq0\}\right)=\mbox{exp}\left(-\D\frac{\nu}{2}\int_0^t\frac{(1-2X_s)^2}{X_s(1-X_s)}ds\right).
$$
Since we have an explicit expression for the transition probability density, we can approximate this distribution by doing
$$
\mathbb{P}\left(\xi> t | X_0=x\right)=\int_0^1p(t;x,y)dy.
$$

We observe that if the process $X_t$ is near the state $1/2$, then there is a small probability that the process is being killed. While if $X_t$ is near 0 or 1, then there is a very high probability that the process is being killed in a next time.

This process can be regarded as a Wright-Fisher model involving only mutation effects with killing. In this case the intensities of mutation are equal and the behavior of the boundary points can be analyzed in the way it is done in pp. 239 of \cite{KT2}. Therefore, since $\nu\geq0$, 0 (and 1) is a \emph{regular} boundary if $0\leq\nu<1/2$, while it is an \emph{entrance} boundary if $\nu\geq1/2$\footnote{We recall that a boundary is said to be \emph{regular} if the process can both enter and leave from the boundary, while it is said to be \emph{entrance} if the boundary cannot be reached from the interior of the state space, but it is possible to consider the process beginning there.}. In Figure \ref{DiffKill1} we can observe this behavior. The picture on the left has $\nu=1/4$, so the boundaries are regular. But when the process is close to 0 or 1, then almost immediately the process is killed. This is not the situation when $\nu=1$ where we have entrance boundaries. It takes more time for the process to be killed and the trajectories can not approach any of the boundary points.

\begin{figure}[h]
\begin{center}
\vspace{-0.4cm}
\includegraphics[height=6.7cm]{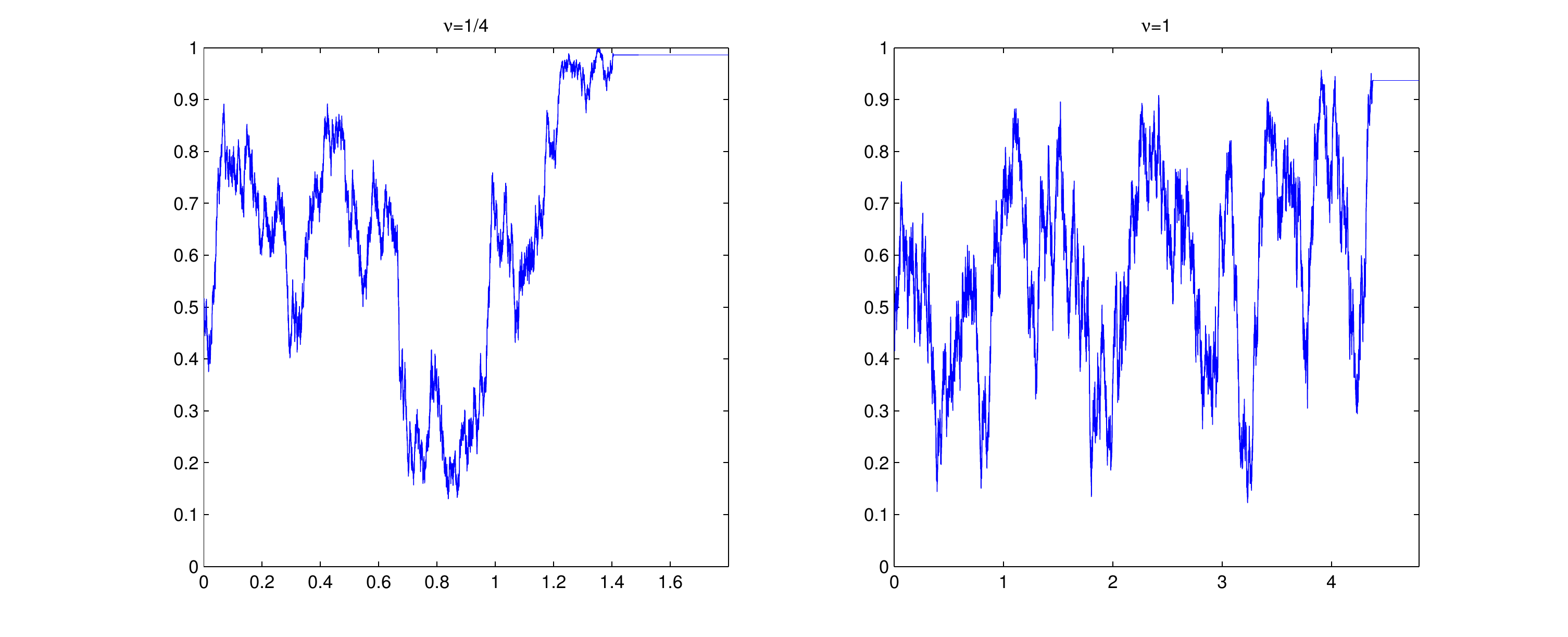}
\vspace{-1cm}
\end{center}
\caption{Trajectories of the diffusion with killing with parameters $\nu=1/4$ (regular boundaries) and $\nu=1$ (entrance boundaries) starting at $x=1/2$.}
\label{DiffKill1}
\end{figure}

\medskip
\noindent {\bf (2)} Let $\{Z_t=(X_t,Y_t), t\geq0\}$ be the switching diffusion process associated with the infinitesimal operator \eqref{DifOpB1}. Now the transition probability distribution is a $2\times2$ matrix-valued function $P(t;x,A)=(P_{ij}(t;x,A))$, defined for every $t\geq0, x\in[0,1]$ and any real Borel set $A$ of $[0,1]$, whose $(i,j)$ entry is given by
$$
P_{ij}(t;x,A)=\mathbb{P}\left(X_t\in A, Y_t=j | X_0=x, Y_0=i\right),\quad i,j\in\{1,2\}.
$$
The density of this matrix-valued distribution (in the sense that $0\leq P(t;x,A)\bm e_2\leq\bm e_2$, for any Borel set $A$) can be described in terms of the matrix-valued orthogonal functions $Q_{n,1}(y)$ in \eqref{QQ} with respect to $W_1(y)$  in \eqref{WWq} (see (3.8) of \cite{dI2}). Therefore
$$
P(t;x,y)=\sum_{n=0}^{\infty}Q_{n,1}(x)\Pi_ne^{\Lambda_{n,1}t}Q_{n,1}^*(y)W_1(y),
$$
where $\Lambda_{n,1}, n\geq0,$ are the (diagonal) eigenvalues \eqref{Eig31} and $\Pi_n^{-1},n\geq0,$ are the (diagonal) norms of the matrix-valued functions $Q_{n,1}(y)$, given by
$$
\Pi_n^{-1}=\|Q_{n,1}^*\|_{W_1}^2=\pi_n^{-1}\begin{pmatrix}1&0\\0&\D\frac{\nu(n+\nu+2)}{4(\nu+1)(n+\nu)}\end{pmatrix},
$$
where $\pi_n^{-1}$ was given by \eqref{norm2q}. It is possible to write $Q_{n,1}(y)$ in terms of the Gegenbauer polynomials (see \cite[Theorem 3.4]{KdlRR}).

The difference of this process with respect to the previous one is that their trajectories can evolve infinitely in time, while the first one has to stop at some random killing time. There are two phases in this process. In the first phase the diffusion evolves as a regular diffusion with infinitesimal operator (see entry (1,1) of $D_1$ in \eqref{DifOpB1})
$$
y(1-y)\partial_y^2+(\nu+1/2)(1-2y)\partial_y,
$$
while in the second phase the diffusion evolves as a regular diffusion with infinitesimal operator
$$
y(1-y)\partial_y^2+\left[(\nu+3/2)(1-2y)-\frac{1}{1-2y}\right]\partial_y.
$$
The description of how the process moves through the two phases is given by the independent coefficient of $D_1$:
\begin{equation}\label{QQ22}
\frac{1}{2y(1-y)}\left(
\begin{array}{cc}
-\nu(1-2y)^2&\nu(1-2y)^2\\
1+\nu&-(1+\nu)
\end{array}
\right).
\end{equation}

It is easy to see that the boundaries 0 and 1 behaves exactly in the same way as in the previous diffusion with killing, i.e. 0 and 1 are \emph{regular} boundaries if $0\leq\nu<1/2$, while they are \emph{entrance} boundaries if $\nu\geq1/2$. Therefore the process is positive recurrent for $\nu\geq1/2$. The important difference now is that in the second phase there is a point in the interior of $[0,1]$ given by $y=1/2$, where the drift coefficient tends to infinity. Therefore we should analyze the behavior of the process near this point (and only if the process is at phase 2). Using the same methodology to study the behavior of boundaries (see pp. 239 of \cite{KT2}) we conclude that the point $1/2$ (both on the left and on the right) is \emph{always} an \emph{entrance} boundary, meaning the the process cannot be reached from the interior of $[0,1/2)$ or $(1/2,1]$ (which depends on the position of the particle when the process starts at phase 2), but it is possible to consider the process beginning at $1/2$. 

This process can also be regarded as a variant of the Wright-Fisher model involving only mutation effects with two different phases. The intensities of mutation are equal and the behavior of the boundaries 0 and 1 in both phases is exactly the same, but, while the process is at phase 2, starting for instance at an interior point of $[0,1/2)$, then there is a force blocking the pass through the threshold located at $1/2$ (same if the interior point is located at $(1/2,1]$). If the process is at phase 1, it can move along the whole state space $[0,1]$ without any restriction at the point $1/2$. This behavior can be seen in Figure \ref{Diff1-2x2}. While the process is at phase 2 (left of the red vertical line) the trajectory is never going to cross the 1/2 horizontal line.

\begin{figure}[h]
\begin{center}
\vspace{-0.5cm}
\includegraphics[height=8.7cm]{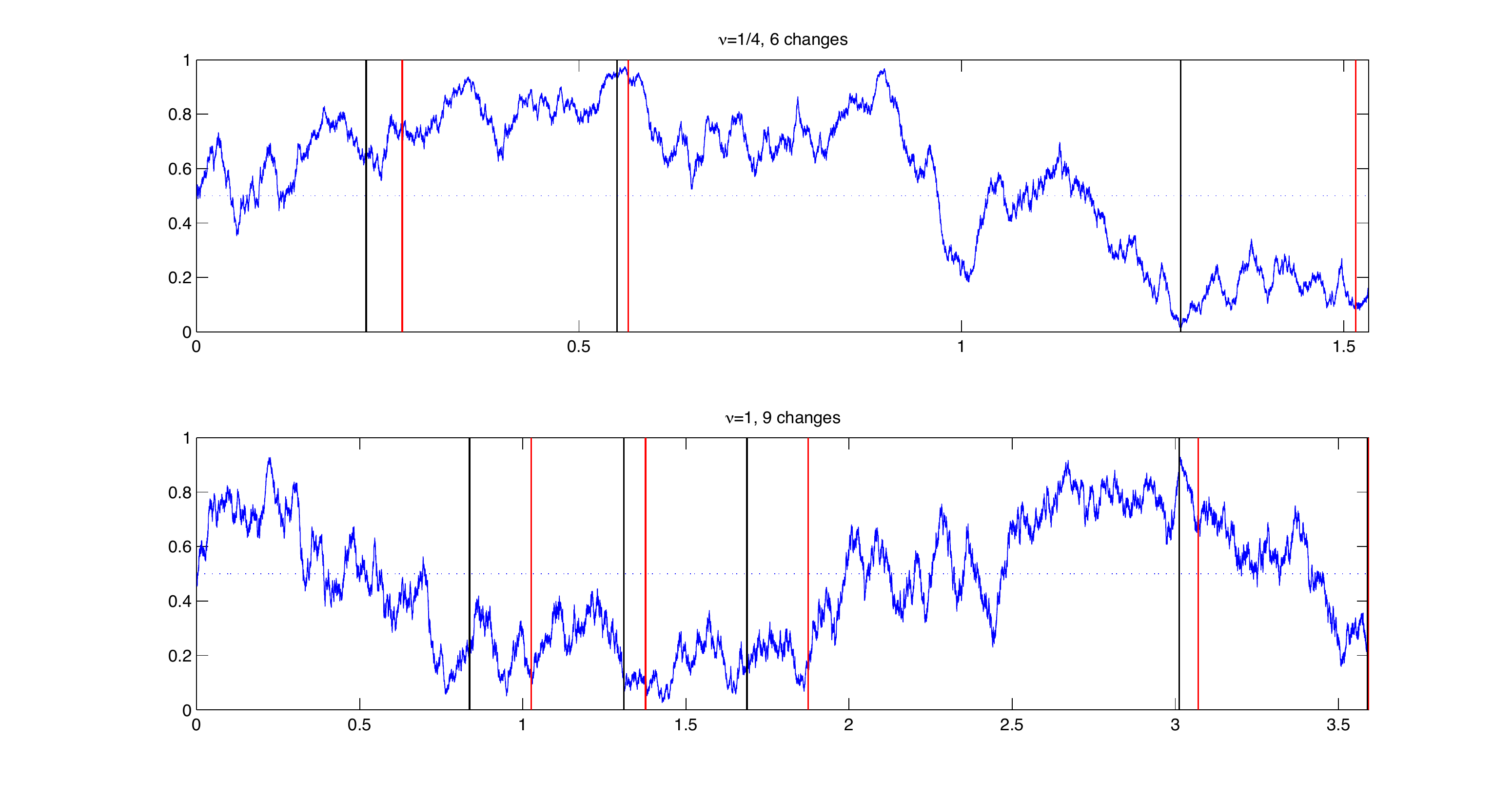}
\vspace{-1.1cm}
\end{center}
\caption{Trajectories of the diffusion with two phases with parameters $\nu=1/4$ (regular boundaries) and $\nu=1$ (entrance boundaries) starting at $y=1/2$ and phase 1. Phase 1 acts on the left of the black vertical line, while phase 2 acts on the left of the red vertical line.}
\label{Diff1-2x2}
\end{figure}

Let us study now how the process moves between the two phases. For that we need to study the matrix \eqref{QQ22} (which is the infinitesimal operator of a continuous-time Markov chain with state space $\{1,2\}$). We observe that  if the process is near 0 or 1, then the diagonal coefficients are very large, meaning that all phases are \emph{instantaneous}, i.e., the waiting times at each phase are very short until the process is far from the boundaries (see again Figure \ref{Diff1-2x2}). We also observe that if the process is near $1/2$ then the entry $(1,1)$ is very small, meaning that phase 1 is \emph{absorbing}, i.e., if the process enters this phase and the position of the particle is close to $1/2$, then it tends to spend long periods of times in that phase (as we can see again in Figure \ref{Diff1-2x2}). At the moment of jumping  from one phase to another, the probabilities are given by the law
\begin{align*}
\mathbb{P}(Y_t=1\to Y_t=2)&=\frac{\nu(1-2y)^2}{\nu(1-2y)^2+1+\nu},\\
\mathbb{P}(Y_t=2\to Y_t=1)&=\frac{1+\nu}{\nu(1-2y)^2+1+\nu}.
\end{align*}
A closer look at these probabilities shows that for all values of $y\in[0,1]$ and $\nu\geq0$ we have
$$
\mathbb{P}(Y_t=1\to Y_t=2)<\mathbb{P}(Y_t=2\to Y_t=1),
$$ 
so that the process tends to stay at phase 1 more time than in phase 2 (a behavior which can be seen again in Figure \ref{Diff1-2x2}).

We finally give an explicit expression of the vector-valued (of dimension 2) invariant distribution $\psi(y)$ given by formula (3.19) of \cite{dI2}, i.e.
\begin{equation*}
\psi(y)=\left(\int_0^1\bm e_2^*W_1(y)\bm e_2dy\right)^{-1}\bm e_2^*W_1(y),\quad \bm e_2^*=(1,1).
\end{equation*}
Since we have an explicit expression of $W_1(y)$ in \eqref{WWq}, we can compute explicitly $\psi(y)$, given in this case by
\begin{equation}\label{invdist}
\psi(y)=\frac{4^\nu\Gamma(\nu+2)[y(1-y)]^{\nu-1/2}}{\sqrt{\pi}(2+\nu)\Gamma(\nu+3/2)}\big(1+\nu\;,\;\nu(1-2y)^2\big).
\end{equation}

In Figure \ref{IDs} we have plotted both components (blue for the first component and red for the second) for the especial cases of $\nu=1/4$ and $\nu=1$. From these plots we clearly see that, for a large time, it is more likely that the process will be in phase 1 than in phase 2, as we previously predicted, especially near the point 1/2, where phase 1 is absorbing.

\begin{figure}[h]
\begin{center}
\vspace{-0.4cm}
\includegraphics[height=6.8cm]{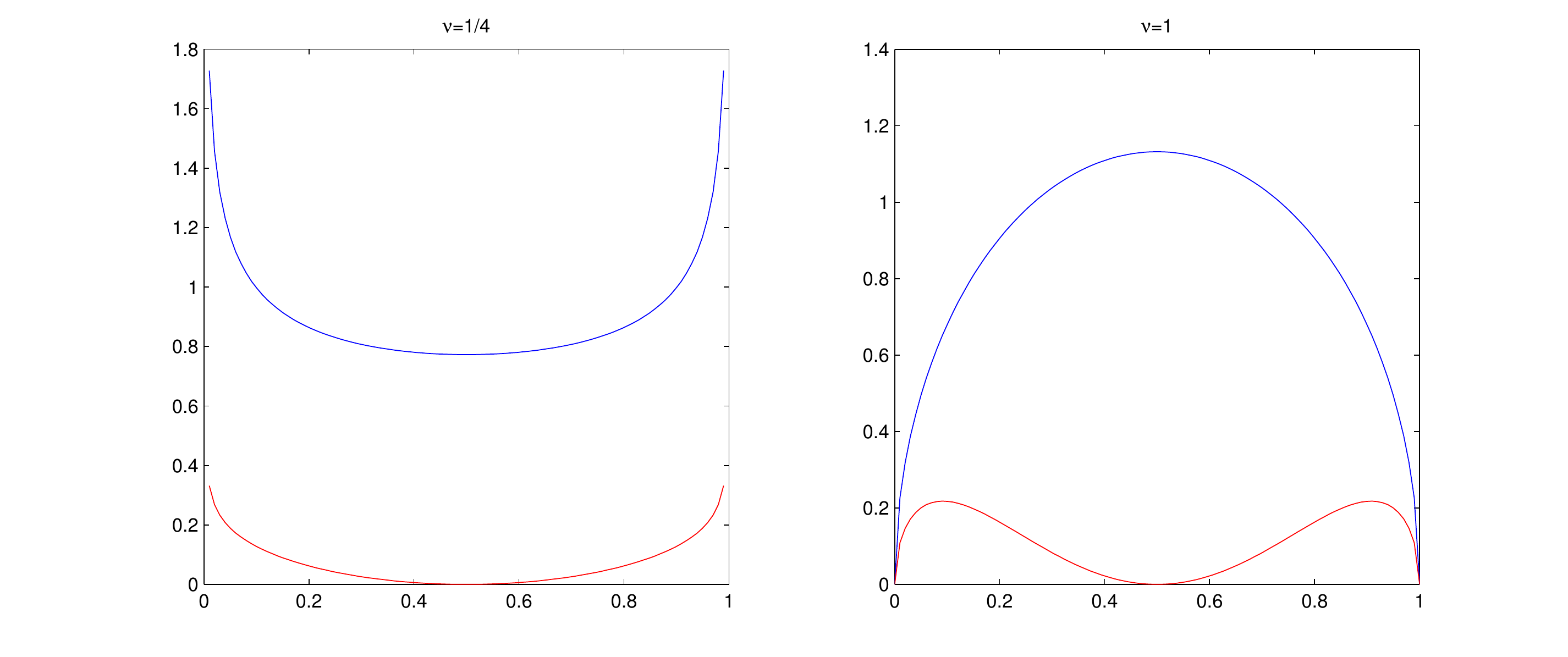}
\vspace{-1cm}
\end{center}
\caption{The components of the vector-valued invariant distribution $\psi(y)$ (in blue the first component and in red the second), for $\nu=1/4$ and $\nu=1$.}
\label{IDs}
\end{figure}

This vector-valued invariant distribution is valid only when the process is positive recurrent, i.e. $\nu\geq1/2$. For $0\leq\nu<1/2$, \eqref{invdist} is also meaningful, but the boundary points of the process are now absorbing, meaning that the correct vector-valued invariant distribution of such cases involves mass jumps at the boundaries 0 and 1 plus a density portion of the form \eqref{invdist}.

\section{The $\ell=2$ case}\label{SEC4}

For the $\ell=2$ case, the matrix $Y$ in \eqref{YY} is given by
\begin{equation*}\label{YY5}
Y=\frac{1}{\sqrt{2}}\begin{pmatrix} 1&0&0&0&1 \\ 0&1&0&1&0 \\0&0&\sqrt{2}&0&0\\0&-1 &0&1&0\\-1&0 &0&0&1 \end{pmatrix}.
\end{equation*}
The weight matrix $W$ and the matrix-valued orthogonal polynomials can now be divided into two examples of size $3\times3$ and $2\times2$, respectively. As in the previous case, we can study two different stochastic models. The first one comes from the coefficients of the three-term recurrence relations, in which case we will have two \emph{continuous-time} level-dependent quasi-birth-and-death processes (with 3 and 2 phases, respectively). These processes are similar to the ones studied in the previous section, i.e. they are rational variations of one-server queues where the interaction between them is remarkable in the first states of the queue. So we will not give any details in this section.

More remarkable is the situation in relation with switching diffusion processes. The conjugation given by the matrix $S(y)$ in \eqref{SS} (see also \eqref{eq:Psi0_k_y} and \eqref{eq:coef_ch}) is
$$
S(y)=\begin{pmatrix} 1 & 0& 0&0&0\\0&1-2y&0&0&0\\0&0&1-\frac{8}{3}y(1-y)&0&0\\0&0&0&1-2y&0\\0&0&0&0&1\end{pmatrix}.
$$
The transformation matrix $T$ is given as in the $\ell=1$ case in \eqref{TTT}.

The first process is a switching diffusion process with state space $[0,1]\times\{1,2,3\}$. The infinitesimal operator is given by

\begin{align*}
D_1=&y(1-y)\partial_y^2\\
&+\left(
\begin{array}{ccc}
(\nu+1/2)(1-2y)&0&0\\
0&(\nu+3/2)(1-2y)-\frac{1}{1-2y}&0\\
0&0&(\nu+5/2)(1-2y)-\frac{6(1-2y)}{3-8y+8y^2}
\end{array}
\right)\partial_y\\
&+\frac{1}{y(1-y)}\left(
\begin{array}{ccc}
-\nu(1-2y)^2&\nu(1-2y)^2&0\\
\frac{3+\nu}{4}&\frac{3+\nu+(1+\nu)(3-8y+8y^2)}{4}&\frac{(1+\nu)(3-8y+8y^2)}{4}\\
0&\frac{3(1-2y)^2}{3-8y+8y^2}&-\frac{3(1-2y)^2}{3-8y+8y^2}
\end{array}
\right),
\end{align*}
with eigenvalue
\begin{equation*}
\Lambda_{n,1}=\left(
\begin{array}{ccc}
-4-n(n+2\nu+4)&0&0\\
0&-1-n(n+2\nu+4)&0\\
0&0&-n(n+2\nu+4)
\end{array}
\right),\quad n\geq0,
\end{equation*}
and weight matrix
\begin{equation*}
W_1(y)=\frac{4^{\nu-2}(\nu+4)[y(1-y)]^{\nu-1/2}}{(\nu+1/2)_2}\left(
\begin{array}{ccc}
(\nu+2)_2&0&0\\
0&4\nu(\nu+2)(1-2y)^2&0\\
0&0&\frac{\nu(\nu+1)(3-8y+8y^2)^2}{3}
\end{array}
\right).
\end{equation*}
Observe that the term $3-8y+8y^2$ is always positive for any $y$. This process may be viewed as an extension of the example studied in Section \ref{S32} (2), but with three different phases. The probabilistic interpretation is very similar and we can study without too much difference the behavior at the boundaries points (including the behavior at the point $y=1/2$ in phase 2), how the process moves between phases and the invariant distribution.

For the second process we have a new phenomenon. We have a switching diffusion process \emph{with killing} with state space $[0,1]\times\{1,2\}$. The infinitesimal operator is given by
\begin{align*}
D_2=y(1-y)\partial_y^2+&\left(
\begin{array}{cc}
(\nu+3/2)(1-2y)-\D\frac{1}{1-2y}&0\\
0&(\nu+1/2)(1-2y)
\end{array}
\right)\partial_y\\
&+\frac{1}{y(1-y)}\left(
\begin{array}{cc}
-\D\frac{\nu+3}{4}-\D\frac{(\nu+1)(3-8y+8y^2)}{4}&\D\frac{\nu+3}{4}\\
\nu(1-2y)^2&-\nu(1-2y)^2
\end{array}
\right),
\end{align*}
with eigenvalue
\begin{equation*}
\Lambda_{n,2}=\left(
\begin{array}{cc}
-1-n(n+2\nu+4)&0\\
0&-4-n(n+2\nu+4)
\end{array}
\right),\quad n\geq0,
\end{equation*}
and weight matrix
\begin{equation*}
W_2(y)=\frac{4^{\nu-2}(\nu+2)(\nu+4)[y(1-y)]^{\nu-1/2}}{(\nu+1/2)_2}\left(
\begin{array}{cc}
4\nu(1-2y)^2&0\\
0&\nu+3
\end{array}
\right).
\end{equation*}

The difference of this process with respect to the previous one is that in the first phase the process can be stopped at some random killing time, so the diffusion runs according to the infinitesimal operator 
\begin{equation}\label{2kill}
y(1-y)\partial_y^2+\left[(\nu+3/2)(1-2y)-\frac{1}{1-2y}\right]\partial_y-\frac{(\nu+1)(3-8y+8y^2)}{4}.
\end{equation}
The second phase runs as a regular diffusion with infinitesimal operator
$$
y(1-y)\partial_y^2+(\nu+1/2)(1-2y)\partial_y.
$$
The description of how the process moves through the two phases is given by
\begin{equation*}
\frac{1}{y(1-y)}\left(
\begin{array}{cc}
-\D\frac{\nu+3}{4}&\D\frac{\nu+3}{4}\\
\nu(1-2y)^2&-\nu(1-2y)^2
\end{array}
\right).
\end{equation*}

\begin{figure}[t]
\begin{center}
\vspace{-0.3cm}
\includegraphics[height=8.8cm]{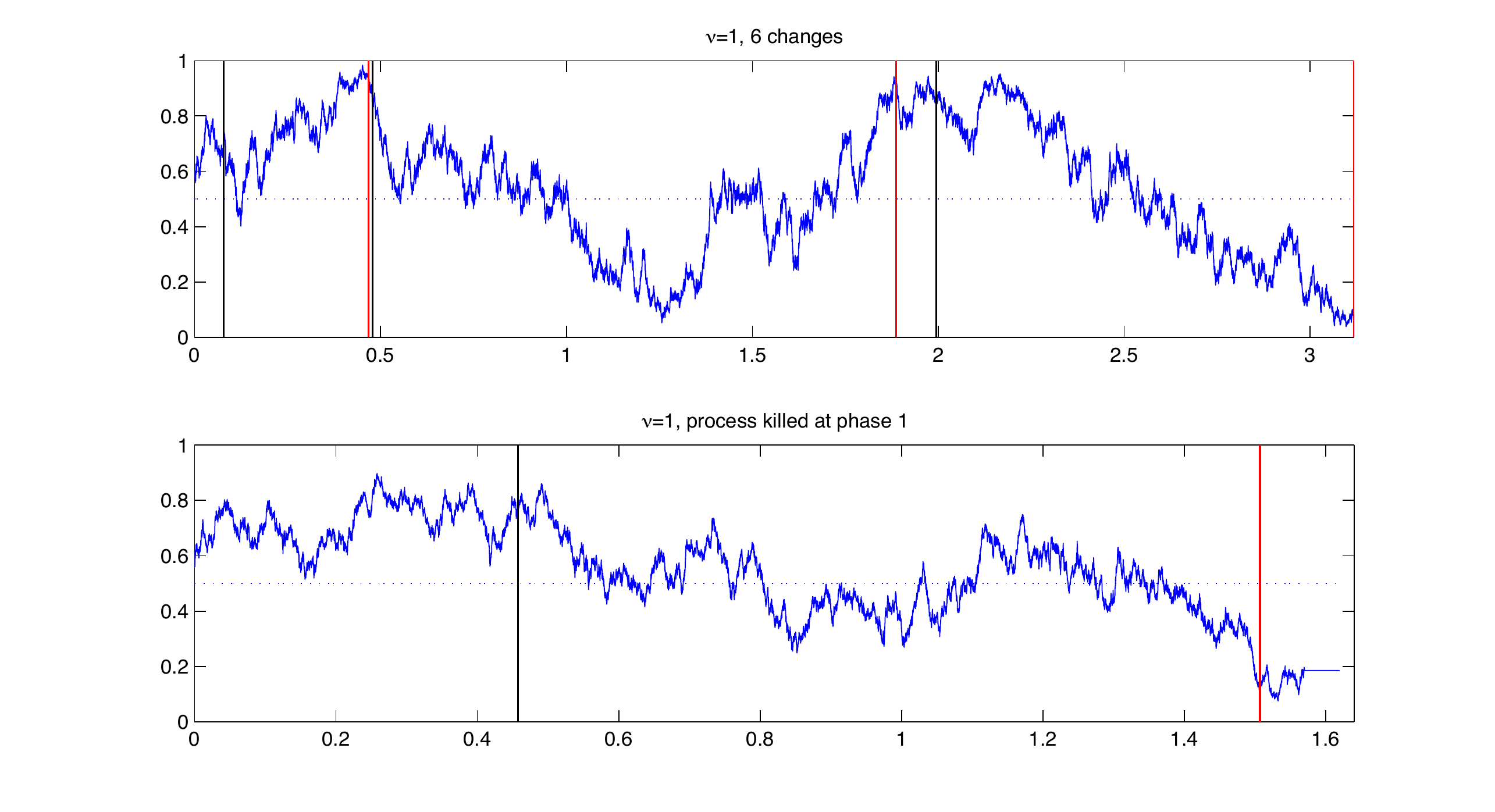}
\vspace{-1.1cm}
\end{center}
\caption{Trajectories of the diffusion with two phases with parameters $\nu=1$ starting at $y=3/5$ and phase 1. Phase 1 acts on the left of the black vertical line, while phase 2 acts on the left of the red vertical line.}
\label{SDPK}
\end{figure}

This process can be regarded as a variant of the Wright-Fisher model involving only mutation effects with two different phases, one of them with a killing factor. The behavior of the boundaries 0 and 1 in both phases is exactly the same, but, while the process is at phase 1, starting for instance at an interior point of $[0,1/2)$, then there is a force blocking the pass through the threshold located at $1/2$ (same if the interior point is located at $(1/2,1]$). Also in this phase the process may terminate according to the killing coefficient given in \eqref{2kill} (see second picture of Figure \ref{SDPK}). If the process is at phase 2, it can move along the whole state space $[0,1]$ without any restriction at the point $1/2$ or being killed (see again Figure \ref{SDPK}). As far as the authors know this is the \emph{first example} of this kind that can be studied explicitly using spectral analysis of the infinitesimal operator.



\begin{thebibliography}{99}


\bibitem{BW} Bhattacharya, R. N. and Waymire, E. C., {\em Stochastic processes with applications}, 
Wiley Series in Probability and Mathematical Statistics: Applied Probability and Statistics. A Wiley-Interscience Publication. John Wiley \& Sons, Inc., New York, 1990.


\bibitem{BM}\textrm{Bakry, D. and Mazet, O.},
\textit{Characterization of Markov semigroups on $\R$ associated to some families of orthogonal polynomials}, S\'eminaire de Probabilit\'es XXXVII, 60--80, Lecture Notes in Math., 1832, Springer, Berlin, 2003.

\bibitem{BL} Barrett, J. F. and Lampard, D. G., {\em An expansion for some second-order probability distributions and its applications to noise problems}, IRE Trans. Information Theory vol. IT-1 (1955), 10--15.

\bibitem{Dette} Dette, H., {\em
First return probabilities of birth and death chains and associated orthogonal polynomials}, Proc. Amer. Math. Soc \textbf{129}, No. 6 (2000), 1805--1815.

\bibitem{DR} Dette, H. and Reuther, B., {\em
Some comments on quasi-birth-and-death processes and matrix measures}, J. Probability and Statistics Volume 2010 (2010), Article ID 730543, 23 pages.

\bibitem{DRSZ} Dette, H., Reuther, B., Studden, W. and Zygmunt, M., {\em Matrix measures and random walks with a block tridiagonal transition matrix}, SIAM J. Matrix Anal. Applic. \textbf{29}, No. 1 (2006), 117--142.


\bibitem{Gra} Graham, R. L., Knuth, D. E. and Patashnik, O., {\em Concrete mathematics. 
A foundation for computer science}, Addison-Wesley Publishing Company, Advanced Book Program, Reading, MA, 1989.

\bibitem{Gri} Griffiths, B., {\em Stochastic processes with orthogonal polynomial eigenfunctions}, J. Comput. Appl. Math. \textbf{233} (2009), 739--744. 

\bibitem{G2}\textrm{Grünbaum, F. A.},
\textit{Random walks and orthogonal polynomials: some challenges}, Probability, Geometry and Integrable Systems, MSRI Publication, volumen \textbf{55}, 2007.

\bibitem{GdI2} Gr\"unbaum, F. A. and de la Iglesia, M. D.,
\textit{Matrix valued orthogonal polynomials arising from group
representation theory and a family of quasi-birth-and-death
processes}, SIAM J. Matrix Anal. Applic. \textbf{30}, No. 2 (2008),
741--761.

\bibitem{GPT1} Grünbaum, F. A., Pacharoni, I. and Tirao, J. A.,
{\em Matrix valued spherical functions associated to the complex
projective plane}, J. Functional Analysis {\bf 188} (2002),
350--441.

\bibitem{GPT3} Gr\"unbaum, F. A., Pacharoni, I. and Tirao, J.
A., {\em Two stochastic models of a random walk in the U($n$)-spherical duals of U($n+1$)}, Ann. Mat. Pura Appl. \textbf{192} (2013), no. 3, 447--473. 



\bibitem{dI1} de la Iglesia, M. D., \emph{A note on the invariant distribution of a quasi-birth-and-death process}, J. Phys. A: Math. Theor. \textbf{44} (2011) 135201 (9pp).

\bibitem{dI2} de la Iglesia, M. D., \emph{Spectral methods for bivariate Markov processes with diffusion and discrete components and a variant of the Wright-Fisher model}, J. Math. Anal. Appl. \textbf{393} (2012), 239--255.

\bibitem{HvP}
Heckman, G. and van Pruijssen, M., \emph{Matrix valued orthogonal polynomials for Gelfand pairs of rank one}, Tohoku Mathematical Journal, to appear.


\bibitem{ILMV}  Ismail, M. E. H., Letessier, J., Masson, D. and
Valent, G., {\em Birth and death processes and orthogonal
polynomials}, in Orthogonal Polynomials, P. Nevai (editor)  Kluwer
Acad. Publishers, 1990, 229--255.

\bibitem{IMcK}  Itô, K. and McKean, H. P. jr., {\em Diffusion processes and their sample paths}, Springer, New York- Heidelberg-Berlin, 1974.


\bibitem{KMc2}  Karlin, S. and McGregor, J., {\em
The differential equations of birth and death processes, and the Stieltjes moment problem}, Trans. Amer. Math. Soc., \textbf{85} (1957), 489--546.

\bibitem{KMc3}  Karlin, S. and McGregor, J., {\em The classification of birth-and-death processes
}, Trans. Amer. Math. Soc., \textbf{86} (1957), 366--400.

\bibitem{KMc4}  Karlin, S. and McGregor, J., {\em
Linear growth, birth and death processes}, J. Math. Mech., \textbf{7} (1958), 643--662.

\bibitem{KMc5}  Karlin, S. and McGregor, J., {\em
Many server queueing processes with Poisson input and exponential service times}, Pacific J. Math., \textbf{8} (1958), 87--118.

\bibitem{KMc6}  Karlin, S. and McGregor, J., {\em Random walks}, IIlinois J. Math., \textbf{3} (1959), 66--81.

\bibitem{KT2}  Karlin, S. and Taylor, H., {\em
A Second Course in Stochastic Processes}, NY: Academic, 1981.

\bibitem{KoekS} Koekoek, R. and Swarttouw, R. F., 
\emph{The Askey-scheme of hypergeometric orthogonal polynomials and its $q$-analogue}, online at \texttt{http://aw.twi.tudelft.nl/\~{}koekoek/askey.html}, Report
98-17, Technical University Delft, 1998. 

\bibitem{KdlRR}
Koelink, E., de los R\'ios, A. M. and Rom\'an, P.,
{\em Matrix-valued Gegenbauer polynomials}, submitted. See arXiv:1403.2938v1.

\bibitem{KvPR} Koelink, E., van Pruijssen, M. and Rom\'an, P.,
{\em Matrix-valued orthogonal polynomials related to ${(\rm SU}(2)\times{\rm SU}(2),{\rm diag})$}, Int. Math. Res. Not. \textbf{24} (2012), 5673--5730.

\bibitem{KvPR2} Koelink, E., van Pruijssen, M. and Rom\'an, P., {\em Matrix-valued orthogonal polynomials related to ${(\rm SU}(2)\times{\rm SU}(2),{\rm diag})$}, II. Publ. Res. Inst. Math. Sci. \textbf{49} (2013), no. 2, 271--312.

\bibitem{KRo} Koelink, E. and Rom\'an, P., {\em Orthogonal vs. non-orthogonal reducibility of matrix-valued measures}, SIGMA Symmetry Integrability Geom. Methods Appl. \textbf{12} (2016), 008, 9 pages.




\bibitem{LaR} Latouche, G. and Ramaswami, V., {\em Introduction to Matrix Analytic Methods in Stochastic Modeling}, ASA-SIAM Series on Statistics and Applied Probability, 1999.


\bibitem{Mak}  Maki, D. P., {\em
On birth-death processes with rational growth rates}, SIAM J. Math. Anal., \textbf{7} (1976), 29--36.

\bibitem{MY} Mao, X. and Yuan, C., {\em Stochastic differential equations with Markovian switching}, Imperial College Press, London, 2006.

\bibitem{McK} McKean, H. P. jr., {\em Elementary solutions for certain parabolic partial differential equations}, Trans. Amer. Math. Soc. \textbf{82} (1956), 519--548.

\bibitem{Neu} Neuts, M. F., {\em Structured Stochastic Matrices of $M/G/1$ Type and Their Applications}, Marcel Dekker, New York, 1989.

\bibitem{DLMF}
NIST Digital Library of Mathematical Functions. 
http://dlmf.nist.gov/ Release 1.0.5 of 2012-10-01. 
Online companion to \cite{Olver:2010:NHMF}.

\bibitem{Olver:2010:NHMF} 
Olver F.~W.~J., Lozier D.~W., Boisvert R.~F., and Clark C.~W. (eds.): 
NIST Handbook of Mathematical Functions. 
Cambridge University Press, New York, 2010. Print companion to \cite{DLMF}.

\bibitem{vPR1} van Pruijssen, M. and  Rom\'an, P., \emph{Matrix-valued classical pairs related to compact
 Gelfand pairs of rank one}, (2013), SIGMA Symmetry Integrability Geom. Methods Appl. \textbf{10} (2014), 113, 28 pages.
 
\bibitem{vPR} van Pruijssen, M. and Rom\'an, P.,
{\em Deformation of matrix-valued orthogonal polynomials},
preprint, 2016.

\bibitem{S} Schoutens, W., {\em Stochastic Processes and Orthogonal Polynomials}, Lectures Notes in Statistics, \textbf{146}, Springer-Verlag, New York, 2000.

\bibitem{TiraPNAS} Tirao, J. A., \emph{The matrix-valued hypergeometric equation}, Proc. Natl. Acad. Sci. USA \textbf{100} (2003), 8138--8141.

\bibitem{vD2}\textrm{van Doorn, E. A.}, \textit{Stochastic monotonicity and queueing applications of birth-death processes}, Lectures Notes in Statistics, 4, Springer-Verlag, 1981.

\bibitem{vD3} van Doorn, E. A., {\em Quasi-stationary distributions and convergence for quasi-stationarity of birth-death processes}, Adv. Appl. Prob., {\bf 23} (1991), 683--700.

\bibitem{WT}  Wong, E. and Thomas, J. B., {\em
On polynomial expansions os second-order distributions}, J. Soc. Indust. Appl. Math., \textbf{10} (1962), 507--516.

\bibitem{YZ} Yin, G. G. and Zhu, C., {\em Hybrid Switching Diffusions. Properties and Applications},  Stochastic Modelling and Applied Probability, 63. Springer, New York, 2010.

\end{thebibliography}
\end{document}